\documentclass{amsart}
\usepackage[utf8]{inputenc}
\usepackage{geometry}
\usepackage{graphicx}
\usepackage{amsmath,amssymb,amsthm}
\usepackage[all]{xy}
\usepackage{color}
\usepackage{comment}

\title{Valuations on the character variety: \\ Newton polytopes and Residual Poisson Bracket}
\author{Julien March\'e}
\address{Sorbonne Universit\'e, IMJ-PRG, 75252 Paris c\'edex 05, France}
\email{julien.marche@imj-prg.fr}
\author{Christopher-Lloyd Simon}
\address{Universit\'e de Lille, CNRS, UMR 8524 - Laboratoire Paul Painlevé, F-59000 Lille, France}
\email{christopher-lloyd.simon@ens-lyon.fr}

\date{}

\newtheorem{definition}{Definition}
\newtheorem{lemma}{Lemma}
\newtheorem{proposition}{Proposition}
\newtheorem{theorem}{Theorem}
\newtheorem{corollary}{Corollary}
\newtheorem{remark}{Remark}
\newtheorem{theoremintro}{Theorem}
\newtheorem{propintro}{Proposition}

\newtheorem*{lemma*}{Lemma}
\newtheorem*{proposition*}{Proposition}
\newtheorem*{definition*}{Definition}

\newcommand{\C}{\mathbb{C}}
\newcommand{\R}{\mathbb{R}}
\newcommand{\Q}{\mathbb{Q}}
\newcommand{\Z}{\mathbb{Z}}

\newcommand{\ML}{\mathrm{ML}}
\newcommand{\MC}{\mathrm{MC}}
\newcommand{\Mm}{\mathcal{M}}
\newcommand{\Oo}{\mathcal{O}}
\newcommand{\Bb}{\mathcal{B}}
\renewcommand{\H}{\mathbb{H}}
\renewcommand{\P}{\mathbb{P}}

\DeclareMathOperator{\Hom}{Hom}
\DeclareMathOperator{\SL}{SL}
\DeclareMathOperator{\tr}{tr}
\DeclareMathOperator{\Ad}{Ad}
\DeclareMathOperator{\Span}{Span}
\DeclareMathOperator{\Card}{Card}
\DeclareMathOperator{\Vol}{Vol}
\DeclareMathOperator{\Mod}{Mod}
\DeclareMathOperator{\ratrk}{\mathrm{rat.rk}}
\DeclareMathOperator{\degtr}{\mathrm{tr.deg}}
\DeclareMathOperator{\NewSet}{\Delta}
\DeclareMathOperator{\Supp}{Supp}

\begin{document}

\maketitle

\begin{abstract}
We study the space of measured laminations $\ML$ on a closed surface from the valuative point of view.
We introduce and study a notion of Newton polytope for an algebraic function on the character variety. We prove for instance that trace functions have unit coefficients at the extremal points of their Newton polytope.
Then we provide a definition of tangent space at a valuation and show how the Goldman Poisson bracket on the character variety induces a symplectic structure on this valuative model for $\ML$. Finally we identify this symplectic space with previous constructions due to Thurston and Bonahon.

\end{abstract}

% \newpage

\section*{Introduction}

\subsection*{The algebra of functions on the character variety}

Let $S$  be a closed oriented surface of genus $g\ge 2$. Its character variety $X$ is the quotient of the space $\Hom(\pi_1(S),\SL_2(\C))$ by the equivalence relation identifying $\rho_1$ and $\rho_2$ if and only if $\tr \rho_1(\gamma)=\tr \rho_2(\gamma)$ for all $\gamma\in \pi_1(S)$. 
By construction, it is an affine variety whose ring of functions $\C[X]$ is generated by the trace functions $t_\gamma \colon \rho \mapsto \tr \rho(\gamma)$ for $\gamma\in \pi_1(S)$.
The function $t_\gamma$ only depends on the conjugacy class of $\gamma$ up to inversion, that is on the free homotopy class of the corresponding unoriented loop.

These trace functions are not algebraically independent: %of course $t_{1} = 2$ is a constant, but more importantly 
the famous identity $\tr(AB)+\tr(AB^{-1})=\tr(A)\tr(B)$ for $A,B\in \SL_2(\C)$ implies, for instance, that if $\alpha$ and $\beta$ represent simple loops intersecting once, then $$t_\alpha t_\beta=t_\gamma+t_\delta$$ where $\gamma$ and $\delta$ are elements in $\pi_1(S)$ representing the simple curves obtained by smoothing the intersection between $\alpha$ and $\beta$ in the two possible ways. 

This phenomenon generalizes as follows. Given a multiloop $\alpha$, that is a multiset $\{\alpha_1,\ldots,\alpha_n\}$ of non-trivial loops $\alpha_i\in \pi_1(S)$, the function $t_\alpha=t_{\alpha_1} t_{\alpha_2}\cdots t_{\alpha_n}$ can be uniquely decomposed as a linear combination  
\begin{equation}\label{decomposition}
t_\alpha=\sum m_\mu t_\mu
\end{equation}
where each $\mu$ is a multicurve, that is a (possibly empty) multiloop represented by pairwise disjoint, simple, non-trivial loops. This means that the set $\MC$ of multicurves indexes a linear basis for the algebra of characters $\C[X]$ which is privileged from the topological viewpoint; it is also invariant under the (algebraic) automorphism group of $\C[X]$, as we proved in \cite{MS_Aut(CV)_2020}.

It is an old problem to understand the algebraic structure of $\C[X]$, whose study has been initiated by Fricke and Vogt in the late 19th century, and revisited in the seventies by the work of Procesi, Horowitz and Magnus among others (see \cite{Magnus} for a review). 
One approach is to investigate the coefficients $m_\mu$ of the functions $t_\alpha$.

In this article, we define the Newton set $\NewSet(t_\alpha)\subset \MC$ of $t_\alpha$, in analogy with the extremal points of the ordinary Newton polytope of a polynomial, as follows.

\begin{definition*}[Newton Set]
For $f=\sum m_\mu t_\mu$ decomposed in the basis of multicurves, we define its support as $\Supp(f)=\{\mu\in \MC, m_\mu\ne 0\}$.

We say that $\mu\in \Supp(f)$ is extremal in $f$ if there exists a multicurve $\xi$ such that $i(\xi,\mu)>i(\xi,\nu)$ for all $\nu\in\Supp(f)$ distinct from $\mu$.

The Newton set $\NewSet(f)$ is the set of extremal multicurves in $f$.
\end{definition*} 

In this definition, $i(\cdot,\cdot)$ denotes the geometric intersection number, and standard properties of measured laminations imply that $\xi$ can be replaced by a simple curve or a measured lamination.
Our first result is the following.

\begin{theoremintro}[Trace functions are unitarity]
\label{ThmA}
For every multiloop $\alpha=\{\alpha_1,\ldots,\alpha_n\}$% with $\alpha_i\in \pi_1(S)\setminus\{1\}$
, the function $t_\alpha$ is unitary in the sense that $m_\mu=\pm 1$ for all $\mu\in \NewSet(t_\alpha)$. 
\end{theoremintro}

To introduce our next result, recall that the algebra of functions $\C[X]$ carries a natural Poisson bracket stemming from the Atiyah-Bott-Weil-Petersson-Goldman symplectic structure on $X$. Following Goldman  \cite{Goldman_symplectic_1984}, for $\alpha,\beta\in \pi_1(S)$ it is given by the formula 
\begin{equation}\label{goldman}\{t_\alpha,t_\beta\}=\sum_{p\in \alpha\cap \beta}\epsilon_p(t_{\alpha_p\beta_p}-t_{\alpha_p\beta_p^{-1}})
\end{equation}
where the sum ranges over all intersection points $p$ between transverse representatives for $\alpha \cup \beta$ and $\epsilon_p$ is the sign of such an intersection, while $\alpha_p, \beta_p$ denote the homotopy classes of $\alpha, \beta$ based at $p$.

Intuitively, our second result says that the Newton polytope of $\{f,g\}$ is contained in the Newton polytope of $fg$, and interprets the coefficients of $\{f,g\}$ at the extremal multicurves of $fg$ in terms of the symplectic structure of the space $\ML$ of measured laminations on $S$.

\begin{theoremintro}[Extremal structure constants for the Poisson Bracket]
\label{ThmB}
Let $\mu$ and $\nu$ be two multicurves.
For $\xi \in \NewSet(t_\mu t_\nu)$ denote $m_\xi$ its coefficient in $\{t_\mu,t_\nu\}$. 

If $m_\xi \ne 0$ then $\xi \in \NewSet(\{t_\mu,t_\nu\})$. In any case,  $m_\xi = \{i_\mu,i_\nu\}(\lambda)$ is the Poisson bracket of the length functions $i_\mu(\cdot)=i(\mu,\cdot)$ on $\ML$ at any generic measured lamination $\lambda$ satisfying $i(\xi, \lambda)=i(\mu\cup\nu, \lambda)$. 
\end{theoremintro}

To rephrase the statement, we may decompose $\ML$ in subsets indexed by $\NewSet(t_\mu t_\nu)$, defined by $E_\xi=\{\lambda\in \ML\,|\, i(\xi, \lambda)=i(\mu\cup\nu, \lambda)\}$ which have disjoint interiors: then the function $\{i_\mu,i_\nu\}$ is constant over $E_\xi$ and equal to the coefficient of $\xi$ in $\{t_\mu,t_\nu\}$. 

Let us illustrate the theorem with the following example. The curves shown in Figure \ref{fig:sphere} satisfy 
$t_\alpha t_\beta=t_{c_1}t_{c_3}+t_{c_2}t_{c_4}-t_{\gamma}-t_{\delta}$ and $\{t_\alpha,t_\beta\}=2t_\delta-2 t_\gamma$, so we find that $\NewSet(t_\alpha t_\beta)=\{c_1\cup c_3,c_2\cup c_4,\gamma,\delta\}$, whereas $\Delta(\{t_\alpha,t_\beta\})=\{\gamma,\delta\}$.

The Newton set of $t_\alpha t_\beta$ decomposes $\ML$ into 4 domains where $i(\alpha\cup \beta,\lambda)$ is equal to the intersection of $\lambda$ with $c_1\cup c_3$ or $c_2\cup c_4$ or $\gamma$ or $\delta$ respectively. In the interior of these domains $\{t_\alpha,t_\beta\}$ takes the values $0,0,-2,2$ respectively. 

\begin{figure}[htbp]
    \centering
    \def\svgwidth{7cm}
    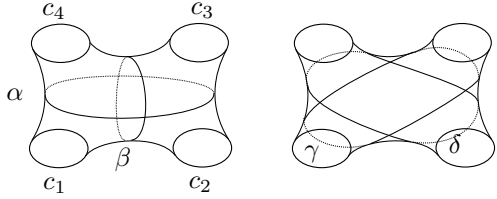
    \caption{Product and Poisson bracket in a sphere with four punctures.}
    \label{fig:sphere}
\end{figure}

Strong relations between the symplectic structures on $X$ and $\ML$ had already been observed, for instance in \cite{Papa-Penner_symplectique-bord_1991} or \cite{Sozen-Bonahon_2001}. However, the relation explained here seems to be new. 

\paragraph{} Beyond these two results, the purpose of this article is to investigate the space of measured laminations from the valuative viewpoint, in particular its symplectic structure. This study was motivated by a new characterisation of valuations associated to measured laminations that we obtained in \cite{MS_Aut(CV)_2020}.
We devote the remaining part of this introduction to an overview of our motivations, as well as the intermediate results that we obtained while revisiting the theory of measured laminations from the valuative viewpoint since we believe they are of independent interest. We take this as an opportunity to recall general ideas for the benefit of a wide audience.

\subsection*{The Newton polytope}

A leading analogy in this article is to think of the collection $(t_\mu)$ as a monomial basis in a polynomial algebra; keeping in mind that it is not stable under multiplication. 

Consider the degree $\deg_d$ defined for $d\in \R^n$ on the algebra $\C[t_1,\ldots,t_n]$ by 
$$\deg_d \left(\sum_\mu m_\mu t^\mu \right)=\max \{\langle m_\mu,d\rangle, m_\mu\ne 0\}$$
where $t^\mu=t_1^{\mu_1}\cdots t_n^{\mu_n}$, and $\langle\cdot,\cdot\rangle$ stands for the usual scalar product. This degree (the opposite of) a monomial valuation.
For $P\in \C[t_1,\ldots,t_n]$, a monomial $t^\mu$ is an extremal point of its usual Newton polytope $\NewSet(P)$ if $m_\mu\ne 0$ and for some $d\in \R^n$ the maximum defining $\deg_d$ is attained uniquely at $t^\mu$.

Our starting point is to replace the degree $\deg_d$ by the valuation associated to a measured lamination $\lambda$ in $S$.
For us a valuation will be a map $v:\C[X]\to \{-\infty\}\cup \R$ satisfying $v(fg)=v(f)+v(g)$ and $v(f+g)\le \max(v(f),v(g))$ for all $f,g\in \C[X]$.
We choose this convention, which is opposite to the usual one, to avoid crowding too many signs.  
Specialists will notice that such valuations centered at infinity of the affine variety $X$ take non-negative values on the ring $\C[X]$ of characters.

In a groundbreaking series of articles starting with \cite{Morgan-Shalen-I_1984}, Morgan-Shalen showed that the character variety $X$ can be compactified using valuations, in the spirit of the Riemann-Zariski compactification. In particular, the space of measured laminations viewed as Thurston's compactification of Teichmüller space, can be embedded in the space of valuations on $\C[X]$ with values in an archimedian group.
However, this embedding used a degeneration process and is not completely explicit: if $v$ is the valuation associated to a lamination $\lambda$, we clearly have $v(t_\mu)=i(\lambda,\mu)$ but it was not clear what should be $v(f)$ for a general element $f\in \C[X]$. 

In our previous article \cite{MS_Aut(CV)_2020}, we showed that the space of measured laminations $\ML$ can be identified with the space of \emph{simple} valuations $v:\C[X]\to \{-\infty\}\cup \R_{\ge 0}$, simple meaning monomial with respect to the multicurve basis:
\begin{equation}\label{max}
v\left(\sum m_\mu t_\mu\right)=\max\{v(t_\mu), m_\mu\ne 0\}.
\end{equation}
This justifies our definition for the Newton set of $f=\sum m_\mu t_\mu$ as the set of $\mu\in \Supp(f)$ such that the maximum in \eqref{max} is attained uniquely at $t_\mu$ for some $v\in \ML$.

For a concrete example, consider the particular case of a multiloop $\alpha$ contained in an incompressible pair of pants $P\subset S$. The subsurface $P$ contains only three simple curves, its boundary components, and they do not intersect each other. Denoting $t_1,t_2,t_3$ the trace functions along these components, we find that $t_\alpha\in \Z[t_1,t_2,t_3]$. This polynomial is often called the Fricke polynomial and has been much studied, see \cite[Section 2.2]{Magnus}. 
Now any valuation associated to a measured lamination on $S$ restricts to a monomial valuation on $\C[t_1,t_2,t_3]$, and we find that our Newton set corresponds to the extremal points of the usual Newton polytope. Even for such $\alpha \subset P$, it is not easy to determine $\NewSet(t_\alpha)$ from the $\alpha_i \in \pi_1(S)$, and the unitarity property is not an obvious one.

We believe that the Newton set $\NewSet(t_\mu t_\nu)$ for two multicurves $\mu$ and $\nu$ is worth studying as it gives some interesting information on the structure constants for the multiplication of $\C[X]$ in the basis of multicurves.
%
% These structures constants measure the extent to which the multicurve basis fails to be multiplicative.
%
As an illustration, we will show the following proposition. For multicurves $\mu,\nu \in \MC$, let $L_\mu(\nu)$ be the multicurve obtained from $\mu \cup \nu$ by smoothing every intersection with a left turn as we travel along a segment of $\mu$ which meets a segment of $\nu$. This ``product'' has been introduced and studied by Luo (see \cite{Luo_simple_2010} and references therein).

\begin{propintro}
For all multicurves $\mu,\nu \in \MC$, if $i(\mu,\nu)>0$ then $L_\mu(\nu)$ and $L_\nu(\mu)$ are distinct, and both belong to $\NewSet(t_\mu t_\nu)$ and $\NewSet(\{t_\mu, t_\nu\})$ simultaneously.
Moreover, the coefficients of $L_\mu(\nu)$ and $L_\nu(\mu)$ in $\{t_\mu,t_\nu\}$ are $\pm i(\mu,\nu)$.
\end{propintro}

It is worth noticing that we only talk about the Newton set and not about the Newton polytope, as we do not know any reasonable notion of convexity in $\ML$.
However, we can define the dual Newton polytope of a function $f\in \C[X]$ as $\NewSet^*(f)=\{v\in \ML \mid v(f)\le 1\}$. % and consider $\{v_\lambda\in \ML \mid i_\lambda (\NewSet^*)\le 1\}$.
Moreover, using the order structure we could define the poset of faces of $\NewSet(f)$, whose combinatorics may be a promising land of investigation but we did not go further in that direction.

\subsection*{Symplectic and combinatorial volumes of dual polytopes}

This paragraph only serves motivational purposes and does not claim new results, it may be skipped harmlessly. 

% For any real function on $\ML$, say the evaluation $v\mapsto v(f)$ of some function $f\in \C[X]$, we define the dual Newton polytope $\NewSet^*(f)\subset \ML$ of a function $f\in \C[X]$ as $\{v\in \ML \mid v(f)\le 1\}$.

Thurston's symplectic form on $\ML$ provides a notion of volume, thus we may ask for the topological meaning of the volume $\Vol \NewSet^*(t_\alpha)$ when $\alpha$ is a multiloop. It vanishes unless $\alpha$ is filling, meaning it intersects every simple curve.
%
% Masur showed that up to scaling, there is only one $\Mod(S)$-invariant measure on $\ML$ in its Lebesgue class. Therefore it is proportional to the Borelian measure which assigns to every open set $U$, the limit $\Vol(U) = \lim_R R^{3\chi}.\Card(RU \cap \MC)$. %as $R\to \infty$ of the number $\Card(R\cdot U \cap \MC)$ of integral measured laminations in the dilated set $r\cdot U$, rescaled by $R^{3 \chi}$.

When $\alpha$ is a filling multiloop, a celebrated theorem of M. Mirzakhani \cite{Mirzakhani_MCGorbits_2016} extended by Rafi and Souto \cite{Rafi-Souto_counting-problems_2017}, estimates the number of elements in its orbit under the modular group $\Mod(S)$ as a bound on their complexity tends to infinity.
More precisely, fix $\beta$ another filling multiloop, and denote $m_g>0$ the volume of the moduli space of hyperbolic metrics on $S$ for the Weil-Petersson form, then :
\[
\lim_{R\to \infty} \frac{\Card\{ \varphi\in \Mod(S) \mid i(\varphi(\alpha), \beta)\le R\}}{R^{6g-6}}=
%\frac{m(\beta)}{m_g}m(\alpha)
\frac{\Vol \NewSet^*(t_\beta) \Vol \NewSet^*(t_\alpha)}{m_g}
\]

% The result holds in the limit where $\beta$ approaches any filling geodesic current, for instance a measured lamination or a hyperbolic metric on $S$.

The identification between measured laminations and simple valuations implies, using Equation \eqref{max}, that the Newton dual polytope $\NewSet^*(f)$ of $f\in \C[X]$ equals the intersection of $\NewSet^*(t_\mu)$ for $\mu \in \NewSet(f)$. These ``elementary cones'' $\{v\in \ML \mid v(t_\mu) <1\}$ are described by explicit sets of linear inequalities in any PL chart of $\ML$, and the volume of their intersection is computable.
%$$m(\alpha)=\Vol\{v\in \ML \mid \forall \mu \in \NewSet(f),\: v(t_\mu)\le 1\}=\Vol \bigcap\limits_{\mu\in \NewSet(t_\alpha)}\{v\in \ML, v(t_\mu)\le 1\}.$$
This yields a constructive procedure to compute Mirzakhani's constant $\Vol \NewSet^*(t_\alpha)$, and shows that it depends only on $\NewSet(t_\alpha)$. It also shows that these volumes are rational.

A different motivation is that this Newton set - as the usual one - could have applications to the problem of counting solutions of algebraic equations in X. We wonder for instance if it helps estimating the number of solutions to a system of $6g-6$ equations $t_{\gamma_i}=x_i$ where $\gamma_1,\ldots,\gamma_{6g-6}\in \pi_1(S)$ and $x_1,\ldots,x_{6g-6}\in \C$. 
This could have interesting applications to $3$-dimensional topology, for instance to evaluate the number of characters of representations of $\pi_1(M)$ of a $3$-manifold $M$ from a Heegaard decomposition.

\subsection*{Measured laminations as valuations}

%There are constructions which allow to go from one description to another: for instance one can easily associate a dual tree to a measured lamination, but the reverse construction is quite difficult. In the same spirit, it is easy to define a length function from the lamination or the tree, but difficult to recover any of those directly from the length function. However, if one knows a valuation $v$ on its whole domain of definition $\C[X]$, the Bass-Serre construction yields an action of the fundamental group on a tree. 

In this article we study measured laminations using the tools of valuation theory. There are two well-known invariants for an archimedean valuation $v$: its rational rank defined as the dimension of the $\Q$-vector space generated by the group $\Lambda_v$ of its values (that is differences of lengths for the corresponding measured lamination), and the transcendence degree of its residue field $k_v$.
These invariants are related by the celebrated Abhyankar inequality $\ratrk(v)+\degtr(k_v)\le 6g-6$. Here we will show the following.
\begin{propintro}[Characterising strict valuations]
For a valuation $v$ associated to a measured lamination $\lambda$, the following properties are equivalent.
\begin{enumerate}
    \item Distinct multicurves $\mu$ and $\nu$ have distinct lengths: $i(\lambda,\mu)\ne i(\lambda,\nu)$.
    \item The residue field of $\C(X)$ at $v$ has transcendence degree $1$, or $k_v=\C$.
    \item The $\Q$-vector space generated by the set of lengths $i(\lambda,\mu)$ for $\mu\in \MC$ has dimension $6g-6$.
\end{enumerate}
\end{propintro}
The first property implies that $v$ defines a total order on the set of multicurves, so the $\max$ in Equation \ref{max} will always be strict, which is why they deserve to be called strict valuations.
They played a prominent role in our previous article, where we showed that almost all valuations are strict in the measure theoretical sense.
They will be equally important in this paper, as the second property enables to define the residual value at $v$ of a function $f\in \C(X)$ satisfying $v(f)\le 0$. 
Combined with the last property, it shows that strict valuations are Abhyankar in the sense that his inequality is an equality: we wonder whether any measured lamination gives rise to an Abhyankar valuation.

We have not come across strict valuations in the literature. Instead we encounter maximal measured lamination, those whose support cannot be enlarged. In this article, we characterize the valuations associated to maximal laminations as being \emph{acute}: for any $\alpha,\beta\in \pi_1(S)\setminus\{1\}$ we never have $v(t_\alpha t_\beta)=v(t_{\alpha\beta})=v(t_{\alpha\beta^{-1}})$ so that these quantities are the lengths for the edges of an acute isosceles triangle. 
We will show that a valuation is acute if and only if any time we smooth a self-intersection of a multiloop which is taut (minimally intersecting in its homotopy class), the two resulting multiloops have distinct $\lambda$-lengths. This property plays a crucial role in the proof of the unitarity theorem. We also show that any strict valuation is acute and wonder if the reciprocal statement is true.

\subsection*{Tangent spaces and Thurston's symplectic structure}

The space of measured laminations is a PL manifold but does not carry any sensible smooth structure (for which intersection numbers have smooth variations), so there is no symplectic structure in the usual sense. However, Thurston showed that most points (maximal laminations) have a well-defined tangent space endowed with a non-degenerate skew-symmetric form. 

In this article we propose a straightforward notion for the tangent space $T_v\ML$ at a valuation, and show that when $v$ is strict, it coincides with the space $\Hom(\Lambda_v,\R)$ which has dimension $\ratrk(v)=6g-6$. 
Then we show how the Goldman Poisson bracket induces a ``residual Poisson bracket" at any strict valuation $v$, thus endowing $T_v\ML$ with a symplectic structure. For future reference we shall name this model after Goldman.
This uses the crucial fact that, given $f,g\in \C[X]$, the Newton polytope $\NewSet(\{f,g\})$ is included in $\NewSet(fg)$ as we already noticed in the second theorem. This property amounts to the inverse inclusion of the dual polytopes $\NewSet^*(\{f,g\}) \supset \NewSet^*(fg)$, which can be written simply as $v(\{f,g\})\le v(fg)$ for all $v\in \ML$.

Finally, we provide precise identifications between this symplectic vector space and two other existing models in the literature, which we now pass under review.

In Morgan-Shalen work, the key notion allowing to relate measured laminations to valuations is the action of $\pi_1(S)$ on real trees. We like to think of this dynamical point of view as lying in between the two others as in the following schematic table.
\begin{center}
\vspace{5pt}
\begin{tabular}{|c|c|c|}
\hline
Topological/Geometrical & Dynamical & Functional \\
\hline
measured foliation & action of $\pi_1(S)$& valuation\\
measured geodesic laminations & on a real tree  & intersection function\\
\hline
\end{tabular}
\vspace{5pt}
\end{center}

For future reference, we name the symplectic vector spaces appearing naturally from each of those approaches after Thurston, Bonahon and Goldman respectively.

\textit{Goldmans's model:} is given by the residual Poisson bracket on $T_v\ML$ which we introduced briefly, it will be described with more detail in the body of the paper.

\textit{Thurston's model:} viewing $\lambda$ as a maximal measured lamination, one can associate a ramified 2-fold covering $S'\to S$ known as the orientation cover of the lamination. The group $H^1(S',\R)$ splits into a symmetric and antisymmetric part with respect to involution of the covering $S'\to S$. The space $H^1(S',\R)^-$ with the cup-product form is the geometrical model for $T_\lambda \ML$. 

\textit{Bonahon's model:} if we consider a trivalent real tree $T$ with a free and minimal action of $\pi_1(S)$, we can consider the space of functions $c\colon V(T)^2\to \R$ on the set of pairs of trivalent vertices of $T$ which satisfy:
\begin{enumerate}
    \item $c(x,y)=c(y,x)$.
    \item $c(x,y)=c(x,z)+c(z,y)$ if $z$ belongs to the geodesic joining $x$ to $y$.
    \item $c(\alpha x,\alpha y)=c(x,y)$ for all $\alpha\in \pi_1(S)$.
\end{enumerate}
Again, this space has a natural antisymmetric form related to the cyclic orientation of $T$ at every trivalent vertex. It is a variation on the notion of transverse cocycles introduced by Bonahon (see \cite{Bonahon_1996}). 
The identification between Thurston's and Bonahon's model is well-known but all proofs we encountered use auxiliary structures like train tracks. At the end of the article, we provide ``invariant" proofs for the following result. 

\begin{theoremintro}[Symplectomorphisms]
There are natural isomorphisms of symplectic vector spaces between the models of Thurston, Bonahon and Goldman.
\end{theoremintro}

%Of course, the identification between the models of Thurston and Bonahon are not due to us (\cite{Bonahon_1996})); moreover the similarity between these and the Atiyah-Bott-Goldman-Weil-Peterson symplectic structure has been noticed by many authors \cite{Papa-Penner_symplectique-bord_1991,Sozen-Bonahon_2001}). 
In particular we provide a new construction of independent interest reminiscent of Milnor's join construction, which starting from a trivalent real tree, gives a space homotopically equivalent to the covering $S'$.
We may wonder which of these three symplectic identifications persist for more general actions of Fuchsian groups on real trees.

\renewcommand{\contentsname}{Plan of the paper}
\setcounter{tocdepth}{1}
\tableofcontents

\subsubsection*{Acknowledgements}
We wish to thank Chris Leininger and Maxime Wolff for useful discussions around this project, as well as Patrick Popescu-Pampu for reading the introduction.

\section{Background}

\subsection{Algebra of functions on the character variety}

Let $S$ be a closed connected and oriented surface of genus $g\ge 1$. We denote by $X$ the character variety of $S$, which is the algebraic quotient of its representation variety $\Hom(\pi_1(S),\SL_2(\C))$ by the conjugacy action of $\SL_2(\C)$, defined as the spectrum of its algebra of functions:
$$\C[X]=\C[\Hom(\pi_1(S),\SL_2(\C))]^{\SL_2(\C)}.$$

A celebrated result of Procesi gives generators and relations for this algebra (which holds for any finitely generated group).
For $\alpha\in \pi_1(S)$, we note $t_\alpha\in \C[X]$ the trace function given by $t_\alpha([\rho])=\tr \rho(\alpha)$. 

\begin{theorem}[Procesi]
The algebra $\C[X]$ is generated by the $t_\alpha$ for $\alpha\in \pi_1(S)$. The ideal of relations is generated by $t_1-2$ and $t_\alpha t_\beta-t_{\alpha\beta}-t_{\alpha\beta^{-1}}$ for all $\alpha,\beta\in \pi_1(S)$. 
\end{theorem}

\begin{definition}
A \emph{multiloop} on $S$ is a class of continuous maps $f\colon \Gamma \to S$ from compact $1$-dimensional manifolds $\Gamma$ to $S$ which is not homotopic to a constant on any component. We consider it modulo the relation declaring $f$ equivalent to $f'\colon \Gamma'\to S$ when there is a homeomorphism $\varphi:\Gamma\to \Gamma'$ such that $f'\circ \phi$ is homotopic to $f$. 

A \emph{multicurve} is a multiloop which is represented by an embedding, we denote $\MC$ the set of multicurves.
\end{definition}

% By considering the homotopy class of the restriction of $f$ to each component of $\Gamma$, we see that a multiloop is the same as a finite collection of non-trivial conjugacy classes in $\pi_1(S)$, considered up to inversion. For a multiloop $\alpha$ represented by $\alpha_1,\ldots,\alpha_n\in \pi_1(S)$, we set $s_\alpha=(-1)^n\prod_{i=1}^n t_{\alpha_i}$.
%
A multiloop amounts to a finite multiset $\{\alpha_1,\ldots,\alpha_n\}$ of non-trivial conjugacy classes in $\pi_1(S)$ considered up to inversion: we define $t_\alpha=\prod_{i=1}^n t_{\alpha_i}$, in particular $t_\emptyset = 1$.
%
% The geometric intersection and self intersection numbers are well defined for multiloops. 
%
The components of a multicurve must be non contractible, simple and pairwise disjoint. %, and we allow the empty multicurve for which $s_\emptyset = 1$.

% We call \emph{multicurve} on $S$ the isotopy class of an embedded one dimensional submanifold $\mu \subset S$ which is a union of curves homotopic to $\mu_i \in\pi_1(\Sigma)\setminus \{1\}$; that is to say the components of $\mu$ must be non contractible, simple and disjoint, but we allow the empty multicurve. We denote by $\MC$ the set of isotopy classes of multicurves on $S$. We set $t_\mu=\prod t_{\mu_i}\in \C[X(\Sigma)]$, in particular $t_\emptyset=1$.
%

%The trace relation $t_\alpha t_\beta-t_{\alpha\beta}-t_{\alpha\beta^{-1}}$ becomes $s_{\alpha \cup\beta}+s_{\alpha\beta}+s_{\alpha\beta^{-1}}=0$, and 
Applying the trace relation recursively to reduce the number of self intersections in multiloops, one may deduce part of the following theorem \cite{PrSi_sl2-skein_2000}. The linear independance requires more work.

% The previous theorem has the following important consequence \cite{PrSi_sl2-skein_2000}.

\begin{theorem}
The family $(t_\mu)_{\mu\in \MC}$ forms a linear basis of the algebra $\C[X]$.
\end{theorem}

%Moreover, Atiyah-Bott and Goldman showed that $X$ is endowed with a natural symplectic structure. Since $X$ is a singular affine variety, this structure is better described in terms of the Poisson bracket on its algebra of functions.

%\begin{theorem}[Goldman]\label{goldman}
%The Poisson bracket on $\C[X]$ 
%$\{\cdot,\cdot\}:\C[X]^2\to \C[X]$ 
%satisfies for all $\alpha,\beta\in \pi_1(S)$:
%$$\{t_\alpha,t_\beta\}=\sum_{p\in \alpha\cap \beta} \epsilon_p(t_{\alpha_p\beta_p}-t_{\alpha_p{\beta_p}^{-1}})$$
%where the sum ranges over all intersection points $p$ between transverse representatives for $\alpha \cup \beta$ and $\epsilon_p$ is the sign of such an intersection, while $\alpha_p, \beta_p$ denote the homotopy classes of $\alpha, \beta$ based at $p$.
%\end{theorem}

\subsection{Deriving the Poisson algebra from the Kauffman algebra}

The multiplication and the Poisson bracket on $\C[X]$ appear naturally as by-products of the Kauffman algebra $K(S,R)$ over some ring $R$ containing an invertible element $A$. Recall that a banded link in an oriented 3-manifold is an oriented submanifold diffeomorphic to a finite union of annuli.

As an $R$-module, the Kauffman algebra is the quotient of the free module over isotopy classes of banded links $L$ in $S\times[0,1]$, by the sub-module generated by Kauffman's local skein relations $[\bigcirc \cup L]=(-A^2-A^{-2})[L]$ and $[L_\times]=A[L_+]+A^{-1}[L_-]$ where $L_\times,L_+,L_-$ are banded links differing in a ball as shown in Figure \ref{fig:kauffman}.

\begin{figure}[htbp]
    \centering
    \def\svgwidth{6cm}
    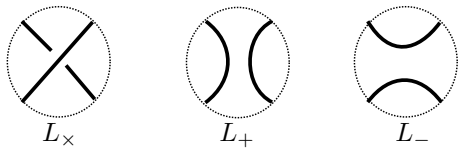
    \caption{Local skein relation.}
    \label{fig:kauffman}
\end{figure}

The product is given by stacking two banded links one above the other. Precisely, $[L_0][L_1]=[\Phi_0(L_0)\cup\Phi_1(L_1)]$ for maps $\Phi_i(x,t)=(x,(t+i)/2)$ of $S\times [0,1]$ into itself. 

Any multicurve $\mu$ on $S$ can be seen as a banded link $[\mu]$ in $S\times[0,1]$ by considering a tubular neighborhood $S\times\{1/2\}$ often called its blackboard framing.

\begin{theorem}[Przytycki/Turaev]\label{skein} Using the previous notations:
\begin{enumerate}
\item The module $K(S,R)$ is a free $R$-module generated by multicurves.
\item The algebra $K(S,\C)$ with $A=-1$ is commutative, and the map sending the blackboard framing $[\mu]$ to $(-1)^{|\mu|}t_\mu$ defines an isomorphism $K(S,\C)\to \C[X]$ where $|\mu|$ denotes the number of components of $\mu$. 
\item The map sending a multicurve to its blackboard framing yields an isomorphism of $\C[A^{\pm 1}]$-modules $K(S,\C)\otimes \C[A^{\pm 1}] \simeq K(S,\C[A^{\pm 1}])$.
%Decomposing in the basis of multicurves yields an isomorphism of $\C[A^{\pm 1}]$-modules $K(S,\C[A^{\pm 1}])\simeq K(S,\C)\otimes \C[A^{\pm 1}]$.
In this setting, we have: $$\{f,g\}=\frac{1}{2}\frac{d}{dA}\Big[ fg-gf\Big]_{A=-1}.$$
\end{enumerate}
\end{theorem}

Let us explain more precisely the first point. Given a diagram $D$ for a banded link $L\subset S\times [0,1]$, we denote by $C$ its set of crossings. For any map $\xi\colon C\to \{\pm 1\}$, set $w_\xi=\sum_{c} \xi(c)$ and consider the diagram $D_\xi$ obtained after smoothing each crossing $c\in C$ according to the sign $\xi(c)$, and removing the $n_\xi$ trivial components which appear in the result. 
We have in $K(S,R)$ the formula:
\begin{equation}\label{smoothing}
[L]=\sum_{\xi:C\to \{\pm 1\}}(-A^2-A^{-2})^{n_\xi} A^{w_\xi} [D_\xi]
\end{equation}
which, after grouping terms corresponding to a same diagram $[D_\xi]$, yields the decomposition of $[L]$ in the basis of multicurves. We use this formula in order to understand the product of two multicurves: intuitively, the product consists in taking the disjoint union and summing over all possible smoothings. 

With the second point, we deduce that the algebra $\C[X]$ has a linear basis indexed by trace functions of multicurves. At $A=-1$, the class of $[L]$ does not change if we change a crossing: we can replace the notion of banded link with the simplest notion of multiloop that we defined previously.

The Kauffman algebra is not completely necessary for our purposes. However, we find it conceptually useful for the following reasons. 
It transforms the trace relation into a local relation whose sign is more convenient (for instance while performing successive diagrammatic computations), and a better understanding of the product in terms of smoothings. It also provides a simple reason to why the Goldman bracket actually satisfies the Jacobi relation: this comes from the third point and the obvious associativity of multiplication the Kauffman algebra.
Finally, in the context of this article, it provides an alternative formula for the Poisson bracket which enlightens Theorem \ref{ThmB}: a smoothing $\xi$ which is extremal for $[\alpha][\beta]$ is also extremal for the Poisson bracket $\{t_\alpha,t_\beta\}$, and its coefficient is $\pm w_\xi$, this integer will be interpreted as a residual Poisson bracket.

\section{Measured laminations and simple valuations}
\subsection{Simple valuations}

It is well-known that a measured lamination $\lambda$ on $S$ is characterized by the length $i(\lambda,\gamma)$ it assigns to any simple curve $\gamma$. 
This ``functional" point of view can be extended to define a map $v_\lambda:\C[X]\to \{-\infty\}\cup [0,+\infty)$
satisfying $v(0)=-\infty$ and for all $f=\sum m_\mu t_\mu$ decomposed in the multicurve basis:
\begin{equation}\label{simple}
v_\lambda(f)=\max\{i(\lambda,\mu) \mid m_\mu\ne 0\}
\end{equation}
where $i(\lambda,\mu)=i(\lambda,\mu_1)+\cdots+i(\lambda,\mu_n)$ for a multicurve $\mu$ with components $\mu_1,\ldots,\mu_n$.
By \cite[Proposition 1.2]{MS_Aut(CV)_2020}, Equation \eqref{simple} is coherent with the fact that for any $\alpha\in \pi_1(S)$, not necessarily simple, we actually have $v_\lambda(t_\alpha)=i(\lambda,\alpha)$. Let us recall \cite[Definition 1.1]{MS_Aut(CV)_2020}.

\begin{definition}\label{defsimple}
A simple valuation on $\C[X]$ is a map $v:\C[X]\to \{-\infty\}\cup \R_{\ge 0}$ satisfying:
\begin{enumerate}
    \item $v(f)=-\infty$ if and only if $f=0$.
    \item $v(fg)=v(f)+v(g)$ for all $f,g\in \C[X]$.
    \item If $f=\sum m_\mu t_\mu$ then $v(f)=\max\{v(t_\mu)\mid m_\mu \ne 0\}.$
\end{enumerate}
\end{definition}
% The last point implies that for any $f,g\in \C[X]$, $v(f+g)\le \max(v(f),v(g))$, with equality if $v(f)\ne v(g)$. 

The following characterization was of fundamental importance in \cite{MS_Aut(CV)_2020}: it yields an homeomorphism between the space of simple valuations and $\ML$, both topologies being defined by simple convergence for the evaluations of multicurves.

\begin{theorem}[M-S] The simple valuations on $\C[X]$ are precisely the $v_\lambda$ for $\lambda\in \ML$.
\end{theorem}

The maximality condition of Definition \ref{defsimple} implies that for any $f,g\in \C[X]$, we have $v(f+g)\le \max(v(f),v(g))$, with equality if $v(f)\ne v(g)$.
Given a multiloop $\alpha$ with a self-intersection $p$, the two smoothings at $p$ give multiloops $\alpha_+$ and $\alpha_-$ and the trace relation reads $t_{\alpha}=\pm t_{\alpha_+}\pm t_{\alpha_-}$.
Hence any valuation $v$ satisfies $v(t_{\alpha})\le \max(v(t_{\alpha_+}),v(t_{\alpha_-}))$.
%Using the isomorphism $K(S,\C)\simeq \C[X]$ of Theorem \ref{skein}, we have such addition relations $[L]=-[L_+]-[L_-]$, so any valuation satisfies $v([L])\le \max(v[L_+],v([L_-])$ with equality if $v([L_+])\ne v([L_-])$. 
%
The following lemma was proven by Dylan Thurston in \cite{Dylan_Thurston_intersection_2009}, and removes the condition $v(t_{\alpha_+})\ne v(t_{\alpha_-})$ for equality to hold. We provide an independent proof in Section \ref{sectiontrees} which relies on the geometry of real trees.

\begin{lemma}[Smoothing Lemma]\label{smoothinglemma}
Let $\alpha$ be a taut multiloop with a self intersection $p$. We denote as usual by $\alpha_+$ and $\alpha_-$ the two smoothings of $\alpha$ at $p$. 
For any $v\in \ML$ we have $v(t_{\alpha})=\max(v(t_{\alpha_+}),v(t_{\alpha_-}))$.
\end{lemma}

% Notice that in this formula we used implicitly the isomorphism $K(S,\C)\simeq \C[X]$ of Theorem \ref{skein}. As we have $[L]=-[L_+]-[L_-]$, any valuation satisfies $v([L])\le \max(v[L_+],v([L_-])$ with equality if $v([L_+])\ne v([L_-])$.

In fact, the inequality $v(t_{\alpha_+})\ne v(t_{\alpha_-})$ holds under certain conditions on $v$ which are satisfied over subsets of full measure in $\ML$, as we now explain. %show that the trivalence condition precisely ensures this inequality.

\subsection{Acute valuations}

We say that a simple valuation $v=v_\lambda\in \ML$ is \emph{positive} if $v(f)>0$ for all non-constant $f\in \C[X]$. It is equivalent to say that $i(\lambda,\alpha)>0$ for all $\alpha\in \pi_1(S)$, or $i(\lambda,\mu)>0$ for all simple curves $\mu$. Such measured laminations are called filling or aperiodic in the literature.

We now introduce the notion of acute valuation, 
%whose denomination will become clear in section \ref{sectiontrees}, and 
which will happen to be equivalent to the notion of maximal measured geodesic lamination.

\begin{definition}\label{acute}
A simple valuation $v\in \ML$ is called acute if it is positive and for any non-trivial $\alpha,\beta\in \pi_1(S)$ we do not have $v(t_{\alpha\beta})=v(t_\alpha t_\beta)=v(t_{\alpha\beta^{-1}})$.
\end{definition}

%\textcolor{red}{Est ce que l'hypothèse de positivité est nécessaire ? Je l'ai rajoutée dans la suivante (le premier item de la réciproque le nécessite), mais est superflu ?}

\begin{lemma}[Unique smoothing]\label{uniquesmoothing}
A positive simple valuation $v_\lambda\in \ML$ is acute if and only if for every taut multiloop $\alpha$, and smoothings $\alpha_\pm$ at a self-intersection, we have:
$$i(\lambda,\alpha_+)\ne i(\lambda,\alpha_-).$$
\end{lemma}

This justifies the terminology: $v\in \ML$ is acute when for every such a multiloop $\alpha$, the values $i(\lambda,\alpha), i(\lambda,\alpha_+), i(\lambda,\alpha_-)$ are the lengths of an acute isosceles triangle such that one of the longest edges corresponds to the multiloop $\alpha$ with greatest self-intersection number.

\begin{proof}
Suppose $v\in \ML$ is acute.
By decomposing $\alpha$ into connected components, we observe that the smoothing concerns at most two of them, and the proof reduces to the following cases. 

\begin{enumerate}
    \item
Either $\alpha$ is a single loop self-intersecting at $p$. Denote by $\gamma,\delta\in \pi_1(S,p)$ the elements such that $\alpha$ is homotopic to $\gamma\delta$. The tautness assumption implies that $\gamma$ and $\delta$ are non-trivial. Depending on the combinatorics of the intersection, one smoothing is homotopic to $\gamma\delta^{-1}$ and the other to the union $\gamma\cup \delta$. If $v(t_{\gamma\delta^{-1}})= v(t_\gamma t_\delta)$ then, from the acute property $v(t_{\gamma\delta})$ differs from them, which contradicts the Smoothing Lemma \ref{smoothinglemma}. 
    \item[2.]
    
Otherwise the multiloop $\alpha$ has two components intersecting at $p$. We denote by $\gamma,\delta\in \pi_1(S,p)$ the (non-trivial) homotopy classes of the two components. Again, $\alpha_+$ and $\alpha_-$ are homotopic to $\gamma\delta$ and $\gamma\delta^{-1}$: the reasoning is the same.
\end{enumerate}

Conversely, suppose $\alpha,\beta \in \pi_1(S)$ are non trivial. 
If they are powers of a same element, say $\alpha=\gamma^n$ and $\beta=\gamma^m$, then $v(t_{\alpha\beta})=\lvert n+m \rvert v(t_\gamma)$ and $v(t_{\alpha\beta^{-1}})=\lvert n-m\rvert v(t_\gamma)$. As $v(t_\gamma)>0$, we have $mn=0$ which is impossible.
% As $v(t_\gamma)>0$, we have $(m+n)^2=(m-n)^2$ thus $4mn=0$ which is impossible.

Consider a hyperbolic structure on $S$, so that $\alpha$ and $\beta$ act on $\tilde{S}\simeq\H^2$ by hyperbolic translations along distinct axes $A_\alpha$ and $A_\beta$ respectively.

\begin{enumerate}
%    \item[2.] If $A_\alpha = A_\beta$ then a power of $\alpha$ is equal to a power of $\beta$, so they are both powers of some primitive element $\gamma$. In that case $\alpha \beta = \gamma^n$ while $\alpha \beta^{-1} = \gamma^{m}$ and $\lvert m \rvert \ne \lvert n\rvert$ since $\alpha,\beta$ are non trivial. Hence the condition $v(t_{\alpha\beta})\ne v(t_{\alpha\beta^{-1}})$ of Definition \ref{acute} is obviously satisfied as $v(t_{\gamma}) \ne 0$ by positivity.
    
    \item If $A_\alpha\cap A_\beta=\{p\}$, then $p$ projects to a point on $\alpha\cap \beta$. The smoothings at $p$ are $\alpha\beta$ and $\alpha\beta^{-1}$. The assumption $i(\lambda, \alpha\beta)\ne i(\lambda, \alpha\beta^{-1})$ says that $v$ satisfies the condition $v(t_\alpha t_\beta)\ne v(t_{\alpha\beta^{-1}})$ ensuring that of Definition \ref{acute}.
    
    \item If $A_\alpha\cap A_\beta=\emptyset$, then up to replacing $\beta$ with $\beta^{-1}$, we may assume the axes point in the same direction. Now, the axes of $\alpha\beta$ and $\beta\alpha$ intersect in a point $p$. This point projects to a self-intersection of $\alpha\beta$ which after smoothing gives alternatively $\alpha\cup \beta$ and $\alpha \beta^{-1}$. The assumption $i(\lambda, \alpha \cup \beta)\ne i(\lambda, \alpha\beta^{-1})$ says that $v$ satisfies the condition $v(t_\alpha t_\beta)\ne v(t_{\alpha\beta^{-1}})$ ensuring that of Definition \ref{acute}.
\end{enumerate}
%
%Then at least one of the homotopy classes $\alpha\beta$, $\alpha\beta^{-1}$, $\alpha \cup \beta$ is not simple.
%
%Thus we may consider a link $L$ with taut projection intersecting at a point for which $[L],[L_+],[L_-]$ are represented, up to sign and permutations, by $t_{\alpha\beta},t_{\alpha\beta^{-1}},t_\alpha t_\beta$. The assumption $v([L_+])\ne v([L_-])$ ensures that one of these trace functions has smaller valuation than the others. 
%\end{proof}
%
%\begin{lemma}
%Let $\alpha,\beta$ be non-trivial elements in $\pi_1(S)$ and $\alpha\cup\beta$ be the multiloop representing them in taut position. We denote by $C(\alpha)$ the conjugacy class of $\alpha\in \pi_1(S)$ and consider the action of $\pi(S)$ on $C(\alpha)\times C(\beta)$ by simultaneous conjugation. Then, the map $\alpha\cap \beta\to C(\alpha)\times C(\beta)/\pi_1(S)$ sending $p\in \alpha\cap \beta$ to the pair $(\alpha_p,\beta_p)$ corresponding to putting the base point at $p$ is a bijection.
\end{proof}

% In the remaining part of this section, we will use the following lemma whose proof is postponed until the next section.
% \begin{lemma}
% The set of acute simple valuations has full measure in $\ML$.
% \end{lemma}

\subsection{Strict valuations}
Given a simple valuation $v$ we can extend it to $\C(X)$ by $v(f/g)=v(f)-v(g)$. We define its valuation ring $\Oo_v=\{f\in \C(X)\mid v(f)\le 0\}$ which has a unique maximal ideal $\mathcal{M}_v=\{f\in \C(X)\mid v(f)<0\}$ and residue field $k_v=\Oo_v/\Mm_v$.

\begin{lemma}\label{strict-residual}
A simple valuation $v=v_\lambda$ satisfies $k_v=\C$ if and only if for all distinct multicurves $\mu,\nu$ we have $i(\lambda,\mu)\ne i(\lambda,\nu)$.
\end{lemma}
Following \cite{MS_Aut(CV)_2020}, we will refer to them as \emph{strict} valuations. We showed that the set of non-strict valuations has zero measure in $\ML$.
\begin{proof}
Suppose that $k_v=\C$ and consider two distinct multicurves $\mu$ and $\nu$. If $v(t_\mu)=v(t_\nu)$ then $t_\mu/t_\nu\in \Oo_v\setminus \Mm_v$ so there exists $\lambda\in \C^*$ such that $t_\mu/t_\nu-\lambda\in \Mm_v$ thus $v(t_\mu/t_\nu-\lambda)<0$. But this means $v(t_\mu-\lambda t_\nu)<v(t_\nu)$, which contradicts the third condition in Definition \ref{defsimple}.

Conversely, suppose that $v$ takes distinct values on distinct multicurves and pick $f=P/Q\in \Oo_v\setminus\Mm_v$. Then $v(P)=v(Q)$: decomposing $P$ and $Q$ in the basis of multicurves, they are necessarily of the form $P=a t_\mu+P'$ and $Q=b t_\mu+Q'$ with $a,b\in \C^*$ and $v(P'), v(Q')<v(t_\mu)$. 
This gives 
$$f=\frac{at_\mu+P'}{bt_\mu+Q'}=\frac{a+P'/t_\mu}{b+Q'/t_\mu}=\frac{a}{b}\mod \Mm_v.$$
\end{proof}

For a simple valuation $v=v_\lambda$, the set of values $\Lambda_v^+ =v(\C[X]\setminus\{0\})$ coincides with $\Lambda^+_v=\{i(\lambda,\mu)\mid \mu \in\MC\}$ by Condition 3 in Definition \ref{defsimple}, and has the structure of an abelian semi-group by Condition 2 in Definition \ref{defsimple}. Its associated group is $\Lambda_v=v(\C(X)^*)$ and consists in differences of $\lambda$-lengths.

%For any simple valuation $v=v_\lambda$ we define a semi-group $\Lambda^+_v=v(\C[X]\setminus\{0\})\subset [0,+\infty)$ and a group $\Lambda_v=v(\C(X)^*)$. As a set, the first one is simply $$\Lambda^+_v=\{i(\lambda,\mu), \mu \text{ multicurve in }S\}.$$

When $v$ is strict, the map $\mu\mapsto i(\lambda,\mu)$ is a bijection between $\MC$ and $\Lambda_v^+$. It is enlightening to think about the semi-group structure on multicurves obtained by pulling back the addition in $\Lambda_v^+$ in the following way.
Let $\mu$ and $\nu$ be two multicurves, viewed as elements of $K(S,\C)$. All smoothings of $\mu\cup \nu$ are multicurves $\xi$ with $i(\lambda,\xi)\le i(\lambda,\mu)+i(\lambda,\nu)$ and equality holds for exactly one of them corresponding to the ``sum of $\mu$ and $\nu$ with respect to $v$".

We define the rational rank of $v$ to be $\ratrk(v)=\dim_\Q \Lambda\otimes \Q$; it satisfies the following Abhyankar inequality (see \cite{Otal_compact-repres_2015} and references therein):
$$\ratrk(v)+\degtr(k_v)\le\dim X$$
from which we deduce that if a simple valuation has maximal rational rank, that is $\ratrk(v)=\dim X$, then it is strict. We will prove the converse in Section \ref{sectionpoisson}.

\section{Newton polytopes of trace functions}

This section relies on the following lemma whose proof is postponed to Section \ref{sectiontrees}.

\begin{lemma}
The set of acute simple valuations has full measure in $\ML$.
\end{lemma}

% In particular the set of strict acute valuations has full measure in $\ML$.

\begin{definition}
Let $v\in \ML$ be any simple valuation and $f \in \C[X]$ any function decomposed as $\sum m_\mu t_\mu$ in the multicurve basis.
\begin{itemize}
\item[-] The multicurve $\mu\in \Supp(f)$ is $v$-extremal in $f$ if  $v(t_\nu)<v(t_\mu)$ for every other $\nu\in \Supp(f)$. 

\item[-] The multicurve $\mu$ is extremal in $f$ if it is $v$-extremal in $f$ for some $v$.

\item[-] The function $f$ is unitary if $m_\mu=\pm 1$ for any extremal multicurve in $f$.

\item[-] The Newton set of $f$ is the subset $\NewSet(f)\subset \MC$ of extremal curves in $f$.

\end{itemize}
\end{definition}
We observe that if $v$ is strict, then $\mu$ is $v$-extremal in $f$ if and only if $v(f)=v(t_{\mu})$. Moreover, as strict valuations are dense in $\ML$, a multicurve is extremal in $f$ if and only if it is $v$-extremal in $f$ for some strict $v$.

\subsection{Trace functions are unitary}

\begin{theorem}[Unitarity]
If $\alpha$ is a multiloop in $S$, then $t_\alpha$ is unitary.
\end{theorem}
\begin{proof}
Let $v$ be a strict acute valuation $v$ and $\mu$ be the unique multicurve such that $v(t_\alpha)=v(t_{\mu})$: we must prove that $m_{\mu}=\pm 1$.
We proceed by induction on the number of intersections of $\alpha$. If there are none, then the result is obvious. Otherwize, put $\alpha$ in taut position and consider its smoothings at an intersection: Lemma \ref{smoothinglemma} and the assumption that $v$ is acute shows that $v(t_{\alpha_+})\ne v(t_{\alpha_-})$. One can suppose that $v(t_\alpha)=v(t_{\alpha_+})>v(t_{\alpha_-})$. The coefficient of $\mu$ in $t_\alpha$ is the same up to sign as in $t_{\alpha_+}$, and the induction hypothesis gives the result. \end{proof}

\begin{remark}
If we represent a taut multiloop as the projection of a banded link $L$ in $S\times[0,1]$, we may decompose it in the basis of multicurves $\mu\in K(S,\Z[A^{\pm}])$ with blackboard framing. 
Then, the coefficient of $\mu$ in $L$ is equal to $A^{n^+-n^-}$ where $n^{\pm}$ count the number of $\pm$-resolutions performed while transforming $L$ into $\mu$.
Putting $A=-1$, we find the sign $(-1)^s$ for the extremal coefficient, where $s$ is the number of self-intersections of $\alpha$. The proof is the same, using inductively the skein relation.
%Using $[L]=-[L_-]-[L_+]$, the proof shows that if $\iota$ counts the number of intersections in a taut diagram for $L$, then its extremal coefficients are $m_\nu = (-1)^\iota$.
\end{remark}

\begin{remark}
We know from \cite{Dylan_Positive_2014} that an ingenious triangular change of basis in $\C[X]$ make the multiplicative structure constants positive. In this basis, the Newton polytope will be the same, and its extremal coefficients will be $1$.
\end{remark}

\begin{corollary}\label{strictimpliesacute}
Any strict valuation is acute. 
\end{corollary}
\begin{proof}
Let $v$ be a strict valuation and consider a taut multiloop $\alpha$. Suppose $v(t_{\alpha_+})=v(t_{\alpha_-})$. Then $t_{\alpha_+}$ and $t_{\alpha_-}$ must have the same $v$-extremal multicurve $\mu$. This defines an open condition on $v\in \ML$. But simple acute valuations are dense in $\ML$ so the same will hold for some acute valuation, contradicting Lemma \ref{uniquesmoothing}. The conclusion follows from the converse part of that lemma.
\end{proof}

% \begin{corollary}
% An acute valuation $v\in \ML$ is strict.
% \end{corollary}

\subsection{Extremal multicurves of $t_\mu t_\nu$ and $\{t_\mu, t_\nu\}$}

Let $\mu$ and $\nu$ be multicurves in $S$ and consider a taut immersion $\mu \cup \nu$ for their union. Note that such an immersion is unique since any two are related by triangle moves \cite{Hass_shortening_1994}, but there are no triangles as the components of $\mu$ and $\nu$ are simple.

We define the embedding $L_\mu(\nu)$ obtained by smoothing all intersections of $\mu \cup \nu$ with a left turn as we travel along a segment of $\mu$ and meet a segment of $\nu$. Smoothing all intersections with a right turn would yield $L_\nu(\mu)$.

This is the product considered by Luo in \cite{Luo_simple_2010}, in particular his {Lemma 8.1} shows that $L_\mu(\nu)$ is a multicurve (it has no trivial components) and his {Theorem 2.1} describes several of its properties.

\begin{proposition} \label{Luo-smoothings-extremal}
Let $\mu,\nu$ be multicurves. The multicurves $L_\mu(\nu)$ and $L_\nu(\mu)$ are extremal for the product $t_\mu t_\nu$, and if $i(\mu,\nu)>0$ they are disctinct.
\end{proposition}

\begin{proof}
If $i(\mu,\nu)=0$ then $L_\mu(\nu)=\mu \cup \nu = L_\nu(\mu)$ and there is nothing more to say. 

Now suppose $i(\mu,\nu)>0$.
We first observe that among all smoothings of the union $\mu \cup \nu$, those which maximise $v_\mu$ are precisely $L_\mu(\nu)$ and $L_\nu(\mu)$. Indeed, we know from \cite[Theorem 2.1 (iii)]{Luo_simple_2010} that $i(\mu,L_\mu(\nu))=i(\mu,\nu)=i(\mu,L_\nu(\mu))$, but any other smoothing $\xi$ is made of segments of $\mu$ and $\nu$ which somewhere alternate between a left turn and right turn, thus forming a bigon with $\mu$ so that $i(\mu,\xi)<i(\mu,\nu)$.
The fact that $L_\mu(\nu) \ne L_\nu(\mu)$ can be obtained from \cite[Corollary 8.2]{Luo_simple_2010} which proves $i(L_\mu(\nu),L_\nu(\mu))=2 i (\mu,\nu)$.

We deduce from the multiplication formula \eqref{smoothing} that the distinct multicurves $L_\mu(\nu)$ and $L_\nu(\mu)$ both appear in the decomposition of $t_\mu t_\nu$ as the unique maximizers of $v_\mu$.
%
% Let $\lambda \in \ML$ belong to the open subset defined by the condition $v_\lambda(L_\kappa) < \max\{v_\lambda(L_\mu(\nu)), v_\lambda(L_\nu(\mu))\}$ for all multicurve $\kappa$ appearing in the product $t\mu t_\nu$ which is distinct from $L_\mu(\nu)$ and $L_\nu(\mu)$.
%
The condition $v_\lambda(L_\mu(\nu)) = v_\lambda(L_\nu(\mu))$ defines on $\lambda \in \ML$ a codimension-$1$ PL-subset (see \cite[Lemma 1.6]{MS_Aut(CV)_2020} for a proof).
Hence a slight perturbation of the valuation $v_\mu$ off that subset in one direction or the other shows that $L_\mu(\nu)$ and $L_\nu(\mu)$ are indeed extremal terms in the product.
\end{proof}

\begin{corollary}
Let $\mu$ and $\nu$ be multicurves such that $i(\mu,\nu)>0$. Then $L_\mu(\nu)$ and $L_\nu(\mu)$ are extremal terms for the Poisson bracket $\{t_\mu, t_\nu\}$ whose coefficients in the basis of multicurves equals $\pm i(\mu,\nu)$.
\end{corollary}

\section{Residual Poisson structure on $\ML$}\label{sectionpoisson}

\subsection{Tangent space}
Recall that $\ML$ embeds in the space of real functions on $\C[X]^*$. We thus define its tangent space at $v$ as the set of maps $\phi=\frac{d}{dt}\big|_{t=0}v_t:\C[X]^*\to \R$, where $v_t$ is a family of simple valuations depending on a parameter $t\in [0,\epsilon[$ starting at $v_0=v$, such that the map $t\mapsto v_t(t_\gamma)$ is differentiable for every curve $\gamma$.

Observe that the pair $(v,\phi):\C[X]^*\to [0,+\infty)\times \R$ satisfies all the axioms in Definition \ref{defsimple} of simple valuations provided the maximum is taken with respect the lexicographic ordering. When $v$ is a strict valuation, the lexicographic ordering depends only on the first coordinate and everything becomes much easier. As we only deal with the strict case, we consider straight away the following as a definition.

\begin{definition}
Let $v\in \ML$ be a strict valuation. We define $T_v\ML$ to be the set of group morphisms $\phi:\C(X)^*\to \R$ satisfying for any function $f \in \C[X]$ decomposed as $f=\sum m_\mu t_\mu$ in the linear basis of multicurves:
\begin{equation}\label{stricto}
\phi(f)=\phi(t_{\nu})\text{ where }\nu\text{ is }v\text{-extremal in }f.
\end{equation}
\end{definition}

We can give equivalent descriptions, several of which serve in the sequel.

\begin{proposition}\label{isomorphisms}
For any strict valuation we have a sequence of natural isomorphisms:
$$T_v\ML=\Hom(\Lambda_v^+,\R)=\Hom(\Lambda_v,\R)=\Hom(\Lambda_v\otimes \Q,\R)=\Hom(\C(X)^*/\Oo_v^\times,\R)$$
where homomorphisms are understood first as semi-group morphisms and then as group morphisms. In particular, $T_v\ML$ has dimension $\ratrk(v)$ (which is $\le \dim X$).
\end{proposition}
\begin{proof}
Recall that the map $\mu\mapsto v(t_\mu)$ is a bijection between the set of multicurves and $\Lambda_v^+$. We have $v(t_\mu)+v(t_\nu)=v(t_\xi)$ where $\xi$ is the $v$-extremal multicurve in $t_\mu t_\nu$. Given $\phi\in T_v\ML$, the map $v(t_\mu)\mapsto \phi(t_\mu)$ gives by construction a morphism of semi-groups $\Lambda_v^+\to \R$ and this construction can easily be reversed, giving the isomorphism $T_v\ML=\Hom(\Lambda_v^+,\R)$. 
The remaining isomorphisms are purely formal, noticing that $\Oo_v^\times$ is the kernel of the group morphism $v:\C(X)^*\to \R$. \end{proof}

%In particular $T_v\ML$ has dimension $\ratrk(v)$, which is at most $\dim X$ by the Abhyankar inequality.

For any $f\in \C(X)$, we can consider the differential of the map $v\mapsto v(f)$ at $v$. We will denote it by $d_v\log f:T_v \ML\to \R$ and define it by $d_v\log f(\phi)=\phi(f)$. We introduced the ``$\log$'' in order to make the following formula look more natural:
$$d_v\log (fg)=d_v\log f +d_v\log g.$$
By Proposition \ref{isomorphisms}, the elements $d_v\log f$ span $T_v^*\ML$: we obtain a basis by letting $f$ range over a family of multicurves whose $v$-lengths form a basis of $\Lambda_v\otimes \Q$. 

\subsection{Residual Poisson structure}

\begin{proposition}\label{inegalitecrochet}
For all $f,g\in \C[X]$ and $v\in \ML$ we have: $v(\{f,g\})\le v(f)+v(g)$.
\end{proposition}
\begin{proof}
By linearity of the Poisson bracket, it is sufficient to prove the inequality for $f=t_\mu$ and $g=t_\nu$ where $\mu$ and $\nu$ are multicurves. Then, by the Leibnitz formula, it is sufficient to prove it for curves $\mu$ and $\nu$. Suppose that $\mu$ and $\nu$ are in taut position and apply Goldman's formula of Theorem \ref{goldman}. It is sufficient to prove that for any $p\in \mu\cap\nu$ we have $v(t_{\mu_p\nu_p}-t_{\mu_p\nu_p^{-1}})\le v(t_\mu t_\nu)$, but this is a consequence of the Smoothing Lemma \ref{smoothinglemma}. 
\end{proof}

Given a strict valuation $v\in \ML$, the preceding proposition allows us to define the residual Poisson bracket at $v$ in the following way. 

\begin{definition}
For $f,g\in \C[X]$ and $v\in \ML$ strict, we define $\{f,g\}_v\in k_v = \C$ by $$\{f,g\}_v=\frac{\{f,g\}}{fg} \mod \Mm_v.$$
\end{definition}

\begin{proposition}
There is an element $\pi_v\in \Lambda^2 T_v\ML$ representing this Poisson structure in the sense that for any $f,g\in \C[X]$ we have 
$$\{f,g\}_v=\langle \pi_v, d_v \log(f)\wedge d_v\log(g)\rangle.$$
\end{proposition}

\begin{proof}
Let us fix $f$ and consider the map $\Psi:\C[X]\to \C$ defined by $\Psi(g)=\{f,g\}_v$. By the Leibnitz identity, this map satisfies $\Psi(g_1g_2)=\Psi(g_1)+\Psi(g_2)$ thus extends to an element of $\Hom(\C(X)^*,\C)$, and we must first show that it vanishes on $\Oo_v^\times$. Any $g\in \Oo_v^\times$ can be written $g=\alpha+h$ with $\alpha\in \C^*$ and $v(h)<0$. We compute 
$$\Psi(g)=\frac{\{f,\alpha+h\}}{f(\alpha+h)}=\frac{\{f,h\}}{f(\alpha+h)}.$$
Since $v(h)<0$ we have $v(f(\alpha+h))=v(f)+v(\alpha+h)=v(f)$ and with Proposition \ref{inegalitecrochet}, $v(\{f,h\})<v(f)$ thus $\Psi(g)\in \Mm_v$ and the claim is proved. 
What we have shown implies that there exists an element $\psi_f\in T_v\ML$ such that $\Psi(g)=\langle \psi_f,d_v \log g\rangle$. As the Poisson bracket is antisymmetric, the same is true with the variables interchanged, and the conclusion follows.
\end{proof}

\section{Actions of $\pi_1(S)$ on real trees}
\label{sectiontrees}

%\subsection{General facts}

A real tree is a metric space $T$ such that any two points $x,y\in T$ are joined by a unique injective segment, and this segment is geodesic.
% Ou bien:
%A real tree is a metric space $T$ such that: any two points $x,y\in T$ are connected by a unique geodesic segment, and the union of two geodesic segments intersecting in one (extremity) point, is a geodesic segment.
%
% Recall that a real tree is a metric space $T$ such that for any $x,y\in T$, there exists a unique arc joining $x$ to $y$, and such that this arc is a geodesic segment.
We consider real trees on which $\pi_1(S)$ acts in a minimal way, meaning without any invariant proper subtree.

The action of an element $\alpha\in \pi_1(S)$ on $T$ either fixes a point and is called elliptic; otherwise it is a hyperbolic translation along an axis $A_\alpha$ with positive translation length $l(\alpha)=\min\{d(x,\alpha x)\mid x\in T\}$, and $d(x,\alpha x)=l(\alpha)$ if and only if $x\in A_\alpha$. 

We face the following alternative. If all of $\pi_1(S)$ acts elliptically, then it has a global fixed point; the minimality assumption implies that $T$ is reduced to a point. If at least one element of $\pi_1(S)$ acts hyperbolically, then the union of all translation axes forms an invariant subtree, which equals $T$ by the minimality assumption.

An action of $\pi_1(S)$ is free if only the trivial element of $\pi_1(S)$ has a fixed point, or equivalently if $l(\alpha)>0$ for all non-trivial $\alpha\in \pi_1(S)$. It is small if the stabilizer of any non-trivial segment in $T$ is cyclic. This condition appears naturally in the following important results. 

\begin{theorem}[Culler-Morgan]
If $T_1,T_2$ are two real trees with a small and minimal action of $\pi_1(S)$, then there is an equivariant isometry $\Phi:T_1\to T_2$ if and only if $l_1(\alpha)=l_2(\alpha)$ for all $\alpha\in \pi_1(S)$.
\end{theorem}

\begin{theorem}
To any measured lamination $\lambda$ in $\ML$ one can associate a ``dual tree" $T_\lambda$ together with a small minimal action of $\pi_1(S)$ on $T_\lambda$. Moreover (Skora theorem), any tree with a small and minimal action of $\pi_1(S)$ is produced in this way. 
\end{theorem}

Let us briefly outline the construction of the dual tree to a measured lamination, in the case where $\lambda$ is filling (or equivalently when the simple valuation $v_\lambda$ is positive).

First represent the filling measured lamination $\lambda$ on $S$ by a measured geodesic lamination for some fixed hyperbolic metric, and lift it in $\tilde{S}$ to obtain a $\pi_1(S)$-invariant measured geodesic lamination $\tilde{\lambda}$. Following \cite[Section 2.3]{Otal_hyperbolization}, the tree $T_\lambda$ is the quotient of $\tilde{S}$ by the equivalence relation whose classes are given either by the closure of a connected component of $\tilde{S}\setminus \tilde{\lambda}$ or by a leaf of $\tilde{\lambda}$ which is not contained in the previous classes. The quotient map $f\colon \tilde{S}\to T_\lambda$ is clearly $\pi_1(S)$-equivariant.

To describe the complement of a point $x\in T_\lambda$, consider its preimage $f^{-1}(\{x\})$. If it consists in a geodesic leaf of $\lambda$, then $T_\lambda\setminus \{x\}$ has $2$ connected components. Otherwise it is isometric to the closure of an ideal hyperbolic polygon with $k$-sides, so $T_\lambda\setminus \{x\}$ has $k>2$ components, and $x$ is called a \emph{branch point} of $T$.
In any case the connected components of $T_\lambda\setminus\{x\}$ have a cyclic orientation which is $\pi_1(S)$-invariant.
These local cyclic orientations match together to give a global cyclic orientation on the Gromov boundary of $T_\lambda$. See \cite{Wolff_compactification_2011} for more details.

The map $f:\tilde{S}\to T_\lambda$ is not proper so does not extend to the Gromov boundary.
A non-trivial element $\alpha\in \pi_1(S)$ acts on $\tilde{S}\simeq \H^2$ by hyperbolic translation along an axis which is transverse to $\lambda$ and thus crosses every leaf at most once. Hence the projection $f$ maps it bijectively to a geodesic in $T$ which, by equivariance, coincides with the axis $A_\alpha$. Hence we can associate to the attractive and repulsive points of $\alpha$ in $\partial \H^2=\partial \pi_1(S)$ the corresponding end points of $A_\alpha$ in $\partial T$. This partially defined map between the Gromov boundaries of $\pi_1(S)$ and $T$ is $\pi_1(S)$-equivariant, orientation preserving and independent of the initial hyperbolic metric.

We recall the following proposition from \cite{Conder-Paulin_Erratum-Gromov-R-trees_2020} that we will use repeatedly.
\begin{proposition}\label{Prop_Paulin}
Let $\gamma,\delta$ be two hyperbolic isometries acting on a real tree $T$ with axes $A_\gamma$ and $A_\delta$. Then one of the following hold.
\begin{enumerate}
    \item If $A_\gamma\cap A_\delta=\emptyset$ then $l(\gamma\delta)=l(\gamma)+l(\delta)+2D$ where $D$ is the distance between $A_\gamma$ and $A_\delta$. 
    \item If $A_\gamma\cap A_\delta\ne \emptyset$, we denote by $D\in [0,+\infty]$ the length of the intersection. 
    \begin{enumerate}
        \item[(i)] If $D>0$ and the translation directions of $\gamma$ and $\delta$ on $A_\gamma\cap A_\delta$ coincide, or if $D=0$ then $l(\gamma\delta)=l(\gamma)+l(\delta)$.
        \item[(ii)] if $D>0$ and the translation directions of $\gamma$ and $\delta$ on $A_\gamma\cap A_\delta$ are opposite, then $l(\gamma\delta)<l(\gamma)+l(\delta)$.
    \end{enumerate}
\end{enumerate}
\end{proposition}
As a first corollary of this proposition, we obtain a proof of the smoothing lemma. 
\begin{proof}[Proof of the Smoothing Lemma \ref{smoothinglemma}]
Let us represent our measured lamination $\lambda$ by an action of $\pi_1(S)$ on a tree $T$. Consider a multiloop $\alpha$ with a self-intersection $p$: we wish to prove that $v(t_\alpha)=\max(v(t_{\alpha_+}),v(t_{\alpha_-}))$. Again, we consider two cases.

In the first case, $\alpha$ consists in a single component that we can write $\alpha=\gamma\delta$ in the fundamental group based at $p$. Then after inspection of all possibilities in Proposition \ref{Prop_Paulin}, we see that $l(\gamma\delta)=\max(l(\gamma)+l(\delta),l(\gamma\delta^{-1}))$, as desired.

In the second case, $\alpha=\gamma\cup\delta$ and we must prove $l(\gamma)+l(\delta)=\max(l(\gamma\delta),l(\gamma\delta^{-1}))$. This holds if and only if the axes of $\gamma$ and $\delta$ intersect (Case 2 of Proposition \ref{Prop_Paulin}). 

Base the fundamental group at $p\in \gamma\cap \delta$ and fix a hyperbolic metric on $S$. Now lift $\gamma$ and $\delta$ in $\tilde{S}$ starting from $p$, to obtain quasi-geodesic curves $\tilde{\gamma}$ and $\tilde{\delta}$ in $\tilde{S}\simeq \H^2$ intersecting transversely at $\tilde{p}$.
Since $\gamma \cup \delta$ is taut, 
% there is no weak bigon between them so
their lifts $\tilde{\gamma}$ and $\tilde{\delta}$ intersect only at $p$ (\cite[see Definition 8 to Lemma 10]{Dylan_Thurston_intersection_2009}). Their end-points should therefore be linked in $\partial\H^2$ with respect to the cyclic ordering. As the same holds for the end-points of $A_\gamma$ and $A_\delta$ in $\partial T$, these axes must also intersect, and the lemma is proved.
% and lift $\gamma$ and $\delta$ to $\tilde{S}$ starting from $\tilde{p}$. We thus obtain two quasi-geodesics $\tilde{\gamma}$ and $\tilde{\delta}$ intersecting transversally at $\tilde{p}$ and mapping to $A_\gamma$ and $A_\delta$ respectively. This shows that these axes do intersect at $f(\tilde{p})$, and the lemma is proved.
\end{proof}

\subsection{Trivalent real trees}
Recall that a point $x$ in a real tree is a \emph{branch point} if $T\setminus\{x\}$ has at least three connected components. We will denote by $V(T)$ the set of branch points of $T$. A real tree is \emph{trivalent} if any branch point disconnects it in three connected components.
\begin{proposition}\label{maxacutriv}
Let $T$ be a real tree with a free minimal action of $\pi_1(S)$ associated with a filling measured lamination $\lambda$ and denote by $v$ the associated positive valuation. 
The following are equivalent:
\begin{enumerate}
\item $v$ is acute
\item $T$ is trivalent
\item $\lambda$ is maximal 
\end{enumerate}
\end{proposition}

A measured geodesic lamination $\lambda$ is called maximal if there is no measured geodesic lamination whose support is strictly bigger; or equivalently if the regions in its complement $S\setminus \lambda$ are isometric to the interiors of ideal hyperbolic triangles.
\begin{proof}
$1\iff 2$. Suppose $T$ is trivalent. Let $\alpha,\beta$ be non-trivial elements in $\pi_1(S)$ and consider their translation axes $A_\alpha,A_\beta\subset T$. From Proposition \ref{Prop_Paulin}, we find that $l(\alpha\beta)=l(\alpha\beta^{-1})=l(\alpha)+l(\beta)$ holds only when $A_\alpha$ and $A_\beta$ meet in exactly one point $(D=0)$, which is forbidden by the trivalence assumption. Thus $v$ is acute.

Conversely, suppose $T$ is not trivalent. Consider a branch point $x\in T$ with valency $k>3$. We denote by $C_1,\ldots,C_k$ the components of $T\setminus\{x\}$. They decompose the Gromov boundary of $T$ into disjoint open subsets $\partial C_1,\ldots,\partial C_k$.
It is known that the set of ends of axes $A_\gamma$ is dense in $\partial T\times \partial T$ for $\gamma\in \pi_1(S)$. One proof consists in considering the sequence of fixed points for the elements $\alpha \beta^n$: the attractive points converge to the image by $\alpha$ of the attractive point of $\beta$ and the repulsive points to the repulsive point of $\beta$. By minimality, the set of repulsive points of all $\beta$'s is dense in $\partial T$, and again by minimality the images of a given attractive point by all $\alpha$'s is dense in $\partial T$.
Thus we can find two axes $A_\alpha, A_\beta$ whose ends are respectively in $\partial C_1\times \partial C_3$ and $\partial C_2\times \partial C_4$. These two axes meet exactly in $x$, and Proposition \ref{Prop_Paulin}, case $2$ $(i)$ $(D=0)$ yields $l(\alpha\beta)=l(\alpha\beta^{-1})=l(\alpha)+l(\beta)$, showing that $v$ is not acute.

$3\iff 2$. Recall from the construction of the dual tree $T$ to a filling measured geodesic lamination $\lambda \subset S$ that the valency of a branch point in $T$ is equal to the number of sides to the corresponding hyperbolic ideal polytope in $\tilde{S}\setminus\tilde{\lambda}$. Hence $T$ is trivalent if and only if $\lambda$ is maximal.
%
%(by the positivity assumption, $\lambda$ does not contain closed geodesics).
%
%Following \cite[Section 2.3]{Otal_hyperbolization}, the vertices of $T$ correspond either to leaves or to complementary regions of $\lambda$. Each subset of $\tilde{S}$ corresponding to a vertex splits the complement either in $2$ if it is a geodesic leaf, or $k>2$ if it is isometric to the interior of an ideal hyperbolic polytope with $k$ sides. The same is true for $T$, so it is trivalent if and only if $\lambda$ is maximal.
\end{proof}

It is well-known that the set of maximal laminations has full measure in $\ML$, see Lemma 2.3 in \cite{Lindenstrauss-Mirzakhani}. Hence Proposition \ref{maxacutriv} implies the following corollary.

\begin{corollary}
The set of acute simple valuations has full measure in $\ML$.
\end{corollary}

\subsection{Bonahon cycles}

Let $T$ be a trivalent real tree with a free and minimal action of $\pi_1(S)$. To define its tangent space in the ``moduli space'' of such objects, imagine the combinatorial structure as being fixed while the distance function $d$ undergoes an infinitesimal deformation. Restricting attention to the variation of the distance between branch points, we obtain a symmetric map $c:V(T)^2\to \R$ which is $\pi_1(S)$-invariant and satisfies $c(x,y)=c(x,z)+c(z,y)$ whenever $z$ belongs to the geodesic joining $x$ to $y$. We will refer to these maps as Bonahon cocycles and introduce them formally using a dual approach.

\begin{definition}
We define the space $\Bb(T)$ as the real vector space generated by pairs $(x,y)$ of elements in $V(T)$ subject to the relations:
\begin{enumerate}
    \item $(x,y)=(y,x)$ for all $x,y\in V(T)$
    \item $(x,y)=(x,z)+(z,y)$ if $z$ belongs to the geodesic joining $x$ to $y$. 
\end{enumerate}
The group $\pi_1(S)$ acts linearly on $\Bb(T)$ by $g(x,y)=(gx,gy)$ and Bonahon cocycles are the elements of $\Hom_\pi(\Bb(T),\R)=\Hom(\Bb(T)_\pi,\R)$ where $\Bb(T)_\pi$ is the space of coinvariants.
\end{definition}

\begin{proposition}\label{prop-pairing}
There is a unique alternating bilinear form $\cdot$ on $\Bb(T)$ such that for all pairs $(x,y),(z,t)\in V(T)^2$ we have:
\begin{enumerate}
    \item $(x,y)\cdot(z,t)=0$ if the geodesics from $x$ to $y$ and from $z$ to $t$ are disjoint.
    \item $(x,z)\cdot(z,y) =\frac \epsilon 2$ if $z$ belongs to the geodesic from $x$ to $y$, where $\epsilon= \pm 1$ is the cyclic order of the components $(h_x, h, h_y)$ of $T\setminus\{z\}$ such that $x\in h_x$, $y\in h_y$. %depending on the cyclic according to whether $z$ lies to the left or to the right of the geodesic.
\end{enumerate}
%Moreover, this form is non-degenerate.
\end{proposition}
\begin{proof}
The intersection of $(x,y)$ and $(z,t)$ is either empty or has the form $(a,b)$. Decomposing $(x,y)$ and $(z,t)$ into segments involving $a$ and $b$ as in Figure \ref{fig:inter-cycles}, we are reduced by bilinearity, to cases 1 or 2. This proves both uniqueness and existence. 
%Suppose that $\alpha\in \Bb(T)$ is a non-zero element in the kernel of this form. We consider all finite sets $W$ such that one can write $\alpha=\sum c_i (x_i,y_i)$ for $x_i,y_i\in W$ and take $W$ minimal with respect to the inclusion of their convex hull in $T$. There is a unique way to write $\alpha$ as a combination of $(x_i,y_i)$ if we restrict to pairs $x_i,y_i$ which are closest neighbors in $T$. Moreover, by minimality, the coefficients of the extremal leaves are non-zero. By intersecting with geodesic getting out of $W$, we find a contradiction. 
\end{proof}
\begin{figure}[htbp]
    \centering
    \def\svgwidth{8cm}
    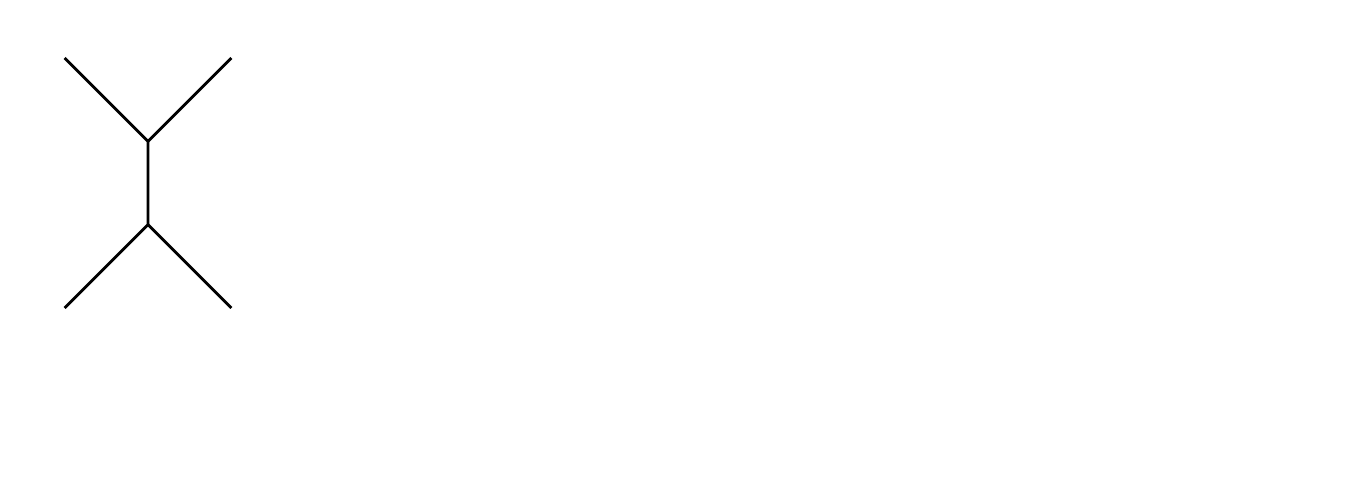
    \caption{}
    \label{fig:inter-cycles}
\end{figure}
It is an amusing exercise to show that this pairing is non-degenerate. Instead we shall deduce it from Poincaré duality in Thuston's model, see Section \ref{sectionthurston}. % This will enable us to deduce non-degeneracy 
Indeed we are interested in the space $\Bb(T)_\pi$ endowed with the following pairing obtained by averaging the previous one, whose non-degeneracy will thus follow from standard arguments in cohomology.

\begin{proposition}
The following sum is finite, and defines an alternating bilinear pairing on $\Bb(T)_\pi$: 
$$(x,y)\cdot_\pi (z,t)=\sum_{g\in \pi_1(S)} (x,y)\cdot g(z,t).$$
\end{proposition}
\begin{proof}
We only have to prove finiteness of the sum. For that, we view $T$ as the dual tree to a maximal measured lamination $\lambda$ on $S$. The vertices $x,y,z,t$ correspond to ideal triangles in $\tilde{S}\simeq \H^2$: choose $x_0,y_0,z_0,t_0$ in each one of them. Since $\pi_1(S)$ acts properly on $\H^2$, the geodesics $[x_0,y_0]$ and $g[z_0,t_0]$ are disjoint for all but a finite number of $g\in \pi_1(S)$. If they are disjoint, they project in the tree either as disjoint geodesics, either as in the middle case of Figure \ref{fig:inter-cycles} and their intersection vanishes.
\end{proof}

% Next we set up an isomorphism between $\Bb(T)_\pi$ and $T_v^*\ML$ which preserves the intersection form.
%
We shall prove in Section \ref{sectionthurston} that $\Bb(T)_\pi$ is the antisymmetric part of $H_1(\tilde{S},\R)$ where $\tilde{S}$ is the orientation covering of the measured lamination $\lambda$, thus recovering Thurston's original point of view on the tangent space $T_\lambda \ML$. 

\subsection{The symplectomorphism theorem}

Fix a strict valuation $v\in \ML$, and recall it identifies the set of multicurves with $\Lambda^+_v$. Let $T$ be a real tree with a free and minimal action of $\pi_1(S)$ representing $v$, that is such that $l(\alpha)=2 v(t_\alpha)$ for all $\alpha\in \pi_1(S)$. 

The next lemma can be deduced from down-to-earth methods (repeated applications of Proposition \ref{Prop_Paulin}). 
It is a direct consequence of a more conceptual construction for $T_v$ using Bass-Serre theory: we refer to \cite[Section 4]{Otal_compact-repres_2015}.
%A more conceptual proof follows from Bass-Serre's construction: we refer to \cite[Section 4]{Otal_compact-repres_2015}.

\begin{lemma}
The distance between two branch points in $T$ belongs to $\Lambda_v$.
\end{lemma}

%\begin{proof}
%Recall that $T$ is the union of the axes $A_\alpha$ for all non-trivial $\alpha\in T$. Following Proposition \ref{Prop_Paulin}, the distance between two disjoint axes and the length of the intersection of two axes belongs to $\Lambda_v$. It is sufficient to prove the lemma for two branch points belonging to a same axis $A_\gamma$. \textcolor{red}{ne faut il pas un argument pour la finitude du nombre de segments ? (il découle du fait qu'il n'y a qu'un nombre fini d'orbites)}  Hence, we suppose that $x\in A_\alpha \cap A_\gamma$ and $y\in A_\beta\cap A_\gamma$. \textcolor{blue}{Then $d(x,y) = $ }
%\end{proof}

Given $\phi\in T_v \ML=\Hom(\Lambda_v,\R)$, we define a corresponding $c_\phi\in \Hom_\pi(\Bb(T),\R)$ by setting $c_{\phi}(x,y)=\frac{1}{2}\phi(d(x,y))$ where $d$ is the distance on $T$. As $d$ is $\pi_1(S)$-invariant, $c$ is also, and the identity $c(x,z)=c(x,y)+c(y,z)$ for $y$ between $x$ and $z$ follows from the triangular equality satisfied by $d$. In other words, there is a well-defined map $\Psi:T_v\ML\to \Hom_\pi(\Bb(T),\R)$ given by $\Psi(\phi)=c_\phi$. 

\begin{proposition}\label{iso1}
The map $\Psi$ induces an isomorphism $T_v\ML\simeq \Hom_\pi(\Bb(T),\R)$. 
\end{proposition}

\begin{proof}
The linearity of $\Psi$ is obvious.
We first prove injectivity: suppose $c_\phi=0$. For any non-trivial $\alpha\in \pi_1(S)$, choose a branch point $x$ on its axis $A_\alpha$ so that the translation length satisfies  $l(\alpha)=2v(t_\alpha)=d(x,\alpha x)$. As $c_\phi(x,\alpha x)=\frac{1}{2}\phi(d(x,\alpha x))$ we get $\phi(v(t_\alpha))=0$, but $\Lambda_v$ is generated by the $v(t_\alpha)$ for $\alpha\in \pi_1(S)$ so $\phi=0$. 

This suggest the construction of the inverse, but this time we think of $\phi$ as a map $\phi:\C(X)^*/\Oo_v^\times\to \R$. Given $c\in \Hom_\pi(\Bb(T),\R)$, we define $\phi(t_\alpha)= c(x,\alpha x)$ for any simple curve $\alpha$, where $x$ is any branch point in $A_\alpha$ (by additivity of $c$, this does not depend on the branch point). We extend $\phi$ to any multicurve by linearity. Finally for any $f\in \C[X]^*$ we set $\phi(f)=\phi(t_\mu)$ where $\mu$ is the $v$-extremal multicurve in $f$. The point is to show that $\phi$ indeed belongs to $T_v\ML$: as it satisfies Equation \eqref{stricto} by construction, it remains to prove that it is multiplicative.

We first show that the defining property $\phi(t_\gamma)=c(x,\gamma x)$ extends to all loops $\gamma \in \pi_1(S)$ by induction on the number of self intersections.
%
%Indeed if $\alpha$ is a non-simple element of $\pi_1(S)$, we have two possibly conflicting definitions of $\phi(t_\alpha)$. 
%
%Let us show that they are coherent by induction on the number of self-intersections. 
%
Suppose $\gamma$ has $n>0$ intersections. Let $p$ be one of them and denote by $\alpha,\beta$ the two elements of $\pi_1(S,p)$ such that $\gamma=\alpha\beta$. Since $v$ is acute, we have either $v(t_{\alpha\beta^{-1}})<v(t_\alpha)+v(t_\beta)=v(t_{\alpha\beta})$ or $v(t_\alpha)+v(t_\beta)<v(t_{\alpha\beta^{-1}})=v(t_{\alpha\beta})$. 

In the first case, the axes $A_\alpha$ and $A_\beta$ intersect along a segment $xy$ and both isometries push $x$ in the direction of $y$ with a translation length greater than $d(x,y)$. 
Then by \cite[Proposition 1.6]{Paulin_Gromov-R-trees_1989}, $l(\alpha\beta)=d(x,\alpha\beta x)=d(x,y)+d(y,\alpha y)+d(\alpha y,\alpha\beta x)$. This gives $c(x,\alpha\beta x)=c(x,y)+c(y,\alpha y)+c(y,\beta x)=c(y,\alpha y)+c(x,\beta x)$. 
Since both $x$ and $y$ belong simultaneously to $A_\alpha$ and $A_\beta$ we establish, by the induction hypothesis, that the two definitions for $\phi(t_\gamma)$ coincide. 

In the second case, the axes $A_\alpha$ and $A_\beta$ are disjoint: let $xy$ be the geodesic joining them. Following again \cite{Paulin_Gromov-R-trees_1989}, $x$ belongs to the axes of both $\alpha\beta$ and $\alpha\beta^{-1}$. By the induction hypothesis, $\phi(t_{\alpha\beta^{-1}})$ is equal to $c(x,\alpha\beta^{-1}x)$. 
%
%One check from the picture that 
The first case in Proposition \ref{Prop_Paulin} shows that $d(x,\alpha\beta^{-1}x)=d(x,\alpha\beta x)$, hence $c(x,\alpha\beta^{-1}x)=c(x,\alpha\beta x)$ and $2\phi(t_{\alpha\beta})=c(x,\alpha\beta x)$ as claimed. 

To finish the proof, consider $f,g\in \C[X]$, we must show that $v(fg)=v(f)+v(g)$. If $\mu$ and $\nu$ are the $v$-extremal multicurves of $f$ and $g$, then the $v$-extremal multicurve of $fg$ is that of $t_\mu t_\nu$, denoted by $\xi$. We must show that $\phi(t_\mu t_\nu)=\phi(t_\xi)=\phi(t_\mu)+\phi(t_\nu)$.
Let us prove more generally that if $\alpha=\alpha_1\cup\cdots\cup\alpha_n$ is a multiloop then $\phi(t_\alpha)=\phi(t_{\alpha_1})+\dots+\phi(t_{\alpha_n})$, reasoning by induction on the self-intersection number of $\alpha$. 

If the components $\alpha_j$ are disjoint, we may replace each one of them by its $v$-extremal smoothing, which remain disjoint, and the result follows from the definition of $\phi$. 
Hence suppose that $\alpha_1$ and $\alpha_2$ intersect at $p$. Up to changing the orientation of $\alpha_2$, we can suppose that $v(t_{\alpha_1\alpha_2})=v(t_{\alpha_1})+v(t_{\alpha_2})$. The computation in the first case at the beginning of the proof shows that $\phi(t_{\alpha_1\alpha_2})=\phi(t_{\alpha_1})+\phi(t_{\alpha_2})$. We also have $\phi(t_{\alpha_1\alpha_2}t_{\alpha_3}\cdots t_{\alpha_n})=\phi(t_{\alpha_1\alpha_2})+\phi(t_{\alpha_3})+\cdots+\phi(t_{\alpha_n})$ by the induction hypothesis.
% The conclusion follows.
\end{proof}

\begin{theorem}\label{iso2}
The isomorphism $\Psi^*:\Bb(T)_\pi\to T_v^*\ML$ preserves the symplectic form.
\end{theorem}

Explicitly, $\Psi^*(x,\alpha x)=d_v\log t_\alpha$ for all $\alpha \in \pi_1(S)$ and any branch point $x\in A_\alpha$. Indeed for all $\phi\in T_v\ML$:  $\Psi(\phi)(x,\alpha x)=\frac{1}{2}\phi(d(x,\alpha x))=\phi(t_\alpha)=d_v\log(t_\alpha)(\phi)$.

\begin{proof}
Let $\alpha,\beta \in \pi_1(S)$ represent two simple curves in $S$. We must prove that $\{t_\alpha,t_\beta\}_v=\langle \pi_v,d_v\log t_\alpha\wedge d_v\log t_\beta\rangle$ equals $(x,\alpha x)\cdot_\pi(y,\beta y)$ for $x\in A_\alpha$ and $y\in A_\beta$.
If $i(\alpha,\beta)=0$ then both quantities are null.
Otherwise, put $\alpha \cup \beta$ in taut position.

We first compute the sum defining $\{t_\alpha,t_\beta\}_v$, in which every intersection $p \in \alpha\cap \beta$ contributes to a term $\epsilon_p(t_{\alpha_p\beta_p}-t_{\alpha_p\beta_p^{-1}}) t_\alpha^{-1}t_\beta^{-1}\bmod{\Mm_v}$. 
The set $\alpha\cap \beta$ is in bijection with pairs of intersecting lifts $(\tilde{\alpha},\tilde{\beta})\subset \tilde{S}\times \tilde{S}$ modulo the diagonal action of $\pi_1(S)$.
%
%Fixing a pair $\alpha_0,\beta_0$ of reference lifts, every pair is represented by some $(\alpha_0,g\beta_0)$ for a unique $g\in \langle \alpha\rangle\backslash\pi/\langle\beta\rangle$. 
%
These lifts correspond bijectively to axes of the form $(A_{\tilde{\alpha}}, A_{\tilde{\beta}})$ in $T$ through the equivariant map $f:\tilde{S}\to T$ which preserves the cyclic orders on the boundaries. %, and the lifts intersect if and only if the corresponding axes intersect.
Fixing representatives $\alpha,\beta \in \pi_1(S)$, every such pair is represented by some $(A_\alpha,gA_\beta)$ for a unique $g\in \langle \alpha\rangle\backslash\pi/\langle\beta\rangle$. 
Using again Proposition \ref{Prop_Paulin}, we can rewrite 
\begin{equation}\label{inter-axes}
\{t_\alpha,t_\beta\}_v=\sum_{g\in \langle \alpha\rangle\backslash\pi/\langle \beta\rangle}\epsilon(A_{\alpha},gA_{\beta})
=\sum_{g\in \langle \alpha\rangle\backslash\pi/\langle \beta\rangle}\epsilon(A_{\alpha},A_{g\beta g^{-1}})
\end{equation}
where $\epsilon(A_\alpha,A_\beta)=\pm 1$ if $A_\alpha$ and $A_\beta$ are like in Figure \ref{or-axes} and $\epsilon(A_\alpha,A_\beta)=0$ in any other configuration. Notice that - as it should - this formula do not depend on the orientations of the axes, but on the local orientation of the tree at the branch points.

\begin{figure}[htbp]
    \centering
    \def\svgwidth{6cm}
    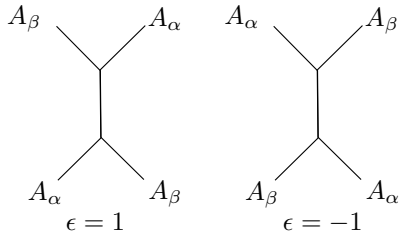
    \caption{Sign rule for the axes.}
    \label{or-axes}
\end{figure}

To end the proof, we fix $x\in A_{\alpha}$ and $y\in A_{\beta}$ to compare Formula \eqref{inter-axes} with $\sum_{g\in \pi} (x,\alpha x)\cdot(gy,g\beta y)$. Grouping them depending on the class of $g$ in $\langle \alpha\rangle \backslash\pi/\langle \beta\rangle$, we are reduced to the following equality, which is easily checked:
$$\epsilon(A_\alpha,A_\beta)=\sum_{m,n\in \Z} (\alpha^n x,\alpha^{n+1} x)\cdot(\beta^m y,\beta^{m+1} y).$$
\end{proof}

% \newpage

\section{Identifying the symplectic tangent models}\label{sectionthurston}

Let us recall Thurston's description for the tangent space to $\ML$ at a maximal measured lamination $\lambda$. We start with an orientation covering $p\colon S'\to S$, which is a ramified covering of degree $2$ with one ramification point in each triangle of the complement $S\setminus \lambda$, and such that the preimage $p^{-1}(\lambda)$ is naturally co-oriented. By the Gauss-bonnet theorem, the set $R$ of ramification points has $4g-4$ elements and the monodromy of the covering is a morphism $\rho\colon \pi_1(S\setminus R)\to\{\pm 1\}$ which is non-trivial around each ramification point. For later purposes, it will be useful to consider the orbifold $S^o$ where ramification points are thought as conical points of order $2$.

Let $H_1(S',\R)^-$ be the antisymmetric part of $H_1(S',\R)$ with respect to the involution of the covering: it supports a non-degenerate antisymmetric form obtained by restricting half the intersection form.
We shall refer to this symplectic space as Thurston's model for $T^*_\lambda\ML$.
We can avoid introducing the covering by considering instead the homology group $H_1(S^o,\R^-)$ with coefficients in the $\pi_1(S^o)$-module $\R$ together the action given by $\gamma.x=\rho(\gamma)x$. The twisted intersection product $H_1(S^o,\R^-)\times H_1(S^o,\R^-)\to H_0(S^o,\R)=\R$ coincides with the previous definition for Thurston's model. We will stick to this point of view in the sequel.

Let $T$ be a trivalent real tree endowed with a free minimal action of $\pi_1(S)$, which is dual to a measured geodesic lamination $\lambda$. %We denote by $S^o$ the orbifold corresponding to $\lambda$. %The orbits of branch points in $T$ correspond either to the connected components of $S\setminus\lambda$, or to the ramification locus $R$ of the covering $S'\to S$. In particular it has $4g-4$ elements. 
%Fix a point inside each triangle (for instance the center of the inscribed circle), to get a collection $R$ of $4g-4$ points in $S$. There exists a $1$-dimensional distribution $\xi$ on $S\setminus R$ which is everywhere tangent to the lamination, and extends to a singular distribution on $S$ with one prong singularity of order 3 in each triangle as in Figure \ref{fig:3-prong}, it is unique up to isotopy. The orientation covering of $\xi$ is a closed surface $S'$ together with a degree $2$ map $S'\to S$ which ramifies over $R$. We denote $S^o$ the orbifold quotient of $S'$ by the Galois involution $\tau$: it has underlying space $S$ and a subset $R$ of conical singularities of order $2$. % Its orbifold fundamental group $\pi_1(S^o)$ maps onto $\pi_1(S)$.
%\begin{figure}[htbp]
%    \centering
%    \includegraphics[width=5cm]{3-prong.png}
%    \caption{Integral curves of $\xi$ in a triangle}
%    \label{fig:3-prong}
%\end{figure}
In the next section we first recover a model for $S^o$ which depends only on $T$: our space will be an infinite dimensional CW-complex homotopic to $S^o$. As a consequence, its fundamental group is canonically attached to $T$ and its homology will be easy to compute from $T$. We will use it extensively to prove that the Bonahon model $\Bb(T)_\pi$ and Thurson model $H_1(S^o,\R^-)$ are naturally isomorphic symplectic vector spaces.

\subsection{A homotopical construction of the orbifold tree}

We first construct a space corresponding to the tree $T$ with an orbifold singularity of order $2$ at every branch point.
As the topology of $T$ induced by the metric is not given by a cell structure, our first task is to build a cellular model of $T$.

\paragraph{Intuition.}
Let us begin with the following analogy: suppose we wish to replace the real line $\R$ with its usual topology, by a CW-complex whose $0$-cells consists in the set $\Q$ of rationals with the discrete topology. We may first add a $1$-cell between every pair of distinct $0$-cells to make the space connected. This creates a $1$-cycle for every triple of distinct rational points, so we attach a $2$-cell to each of those in order to make the space simply connected. Now every $4$-tuple of rationals form the vertices of a $2$-cycle, to which we attach a $3$-cell, and so on.
What we obtain in the limit is Milnor's join construction $E\Q$, which is a space homotopic to $\R$ endowed with a free and proper action of $\Q$. 

We shall play a similar game, replacing $\R$ by the real tree $T$, and $\Q$ by its set of branch points $V(T)$. We first attach a $1$-cell to every pair of distinct branch points. However, we close the triangle $(x,y,z)$ only if $x,y,z \in V(T)$ belong to a same geodesic in $T$. Then we go on similarly in higher dimensions, so that our space will resemble $E\Q$ in restriction to any geodesic of $T$. At this stage, we have a space on which $\pi_1(S)$ acts freely and properly. As it is contractible, its quotient by $\pi_1(S)$ is homotopic to $S$.
Next comes the orbifold singularity: in homotopy theory, this is represented by a $K(\Z/2,1)$-space, that is $\R\P^\infty$. It remains to blow up the preceding construction at every branch point and insert an infinite dimensional space. This construction may look complicated but we shall do it in one shot and few lines below.

\paragraph{Construction.}
A \emph{half-edge} of $T$ is a pair $(x,h)$ consisting in a branch point $x$ of $T$ and a connected component $h$ of $T\setminus\{x\}$; we sometimes just write $h$ as it determines $x$. Let us construct a CW-complex $T^o$ whose $0$-skeleton is the set of half-edges of $T$.
First, we attach a $1$-cell denoted $(h,k)$ between every pair of half-edges incident to the same branch point $x\in V(T)$. 
Now at every branch point $x$, the incident half-edges $h_1,h_2,h_3$ form a triangle homeomorphic to $\R\P^1$ through which we attach a copy of $\R\P^\infty$.
For the moment, $T^o$ is a disjoint union of infinite projective spaces indexed by the set of branched points $V(T)$, we call it the \emph{orbifold part}.

Now, we add a \emph{connecting part}, as suggested in Figure \ref{fig:CW}. 
%
% We now add to $T^o$ \emph{connecting part} modeling the geodesics in $T$. For every pair of distinct branch points $x,y\in V(T)$, we attach the boundary of a $1$-cell $(x,y)$ to the pair half-edges $(x,h)$, $(y,k)$ such that $e$ and $f$ contain the unique geodesic between $x$ and $y$.
%
%Finally, we add to $T^o$ a \emph{connecting part}. %by increasing dimensions from $1$ to $\infty$.
%
Fix $\epsilon>0$ small enough, say $1/3$. Consider a finite set $W$ of branch points $\{x_0, \dots, x_n\}$ aligned on a geodesic of $T$, and denote $h_i,k_i$ the half-edges incident to $x_i$ containing (a non-empty) part of that geodesic. %, which we orient for convenience so that we may order the branch points $x_0, \dots, x_n$. 
The $n$-cell $\NewSet_W=\{(r_x)_{x\in W}\in [0,1-\epsilon]^W, \sum_{x\in W}r_x=1\}$ is a truncated simplex, and there is an obvious inclusion $\Delta_{W'}\subset \Delta_W$ when $W'\subset W$. The face of $\Delta_W$ truncated at $x_i$ corresponds to the set $\Delta_W^{x_i}$ of families $(r_x)$ satisfying $r_{x_i}=1-\epsilon$.
%
%Note $e_i$ and $f_i$ the $0$-cells corresponding to half-edges incident to $x_i$ containing $x_{i-1}$ and $x_{i+1}$ respectively (with $x_{n+1}=x_0$ and $x_{-1}=x_n$).
We attach $\Delta_W^{x_i}$ to the orbifold part of $T^o$ through the map $W\setminus\{x_i\}\to \{h_i,k_i\}$ sending the branch point $x_j$ to the half-edge based at $x_i$ which contains $x_j$, as in Figure \ref{fig:CW}. %
The $1$-cells $\Delta_{\{x,y\}}$ will be called edges and denoted $(x,y)$. % or $(e,f)$ where $e$ and $f$ are the half-edges incident to $x,y$ and containing $y$ and $x$ respectively.
\begin{figure}[htbp]
    \centering
    \includegraphics[width=8cm]{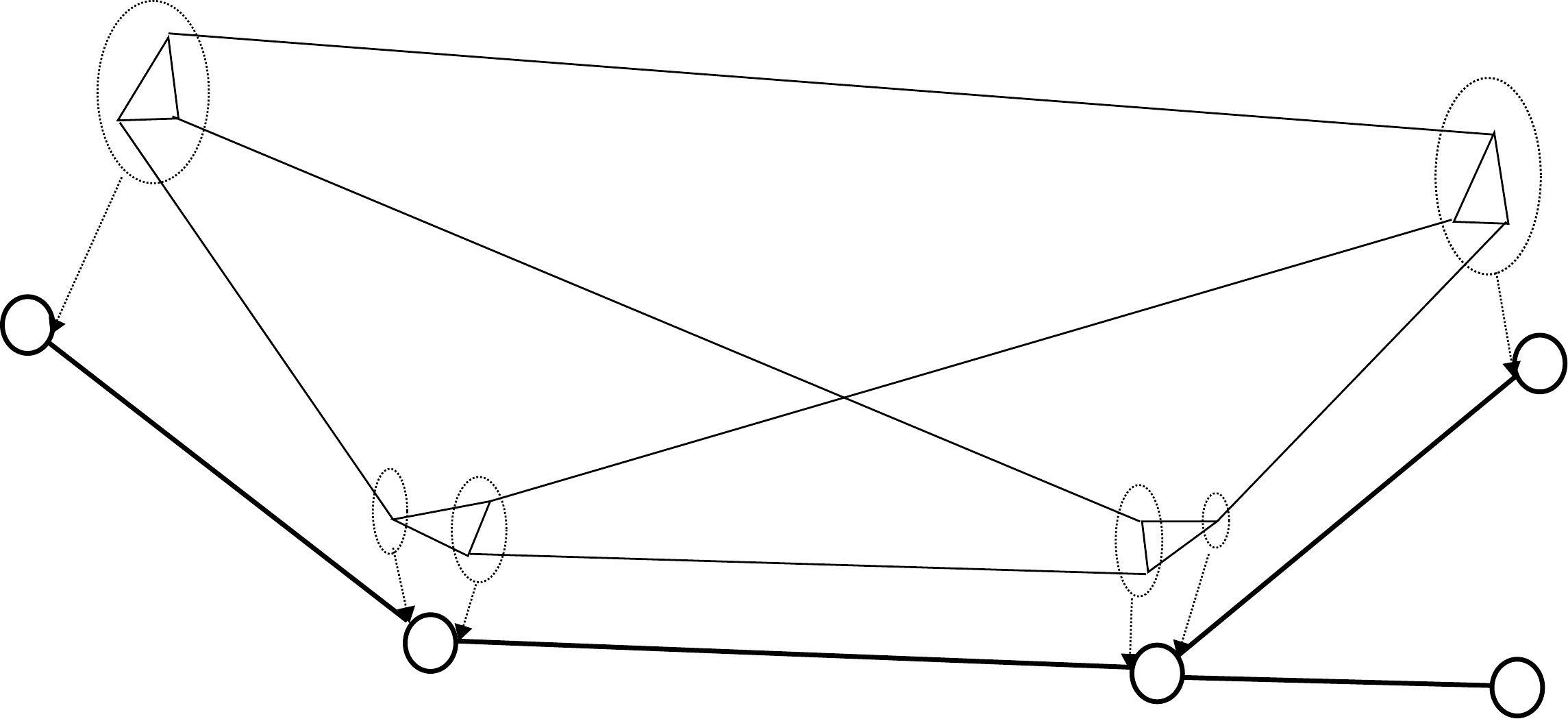}
    \caption{Attaching a 3-cell in $T^o$}
    \label{fig:CW}
\end{figure}

As promised, the action of $\pi_1(S)$ on $T^o$ is now proper so that we may form the quotient $\Sigma^o=T^o/\pi_1(S)$. The following lemma shows that $\Sigma^o$ and $S^o$ are homotopic. Interestingly, the proof consists in constructing an equivariant map $F:T^o\to \tilde{S}^o$, which plays the role of a (non-existing) retraction for the map $f:\tilde{S}\to T$. 

\begin{lemma}\label{equiv-homotopy}
Let $\tilde{S}^o$ be the covering of the orbifold $S^o$ corresponding to the kernel of the natural map $\pi_1(S^o)\to \pi_1(S)$. There exists a $\pi_1(S)$-equivariant map $F:T^o\to \tilde{S}^o$ which induces a homotopy equivalence between $\Sigma^o$ and $S^o$.
\end{lemma}
\begin{proof}

To define $F$, represent $T$ as the dual tree to a measured geodesic lamination $\lambda$ and consider the collection of circles inscribed in each triangle of the complement $S\setminus\lambda$: they lift to a collection of circles $C_x$ in $\tilde{S}\simeq \H^2$ indexed by $x\in V(T)$.
Moreover, the half-edges incident to $x$ correspond bijectively to the three intersection points of $C_x$ with the leaves of the lamination, see Figure \ref{fig:equivmap}.

\begin{figure}[htbp]
    \centering
    \includegraphics[width=6cm]{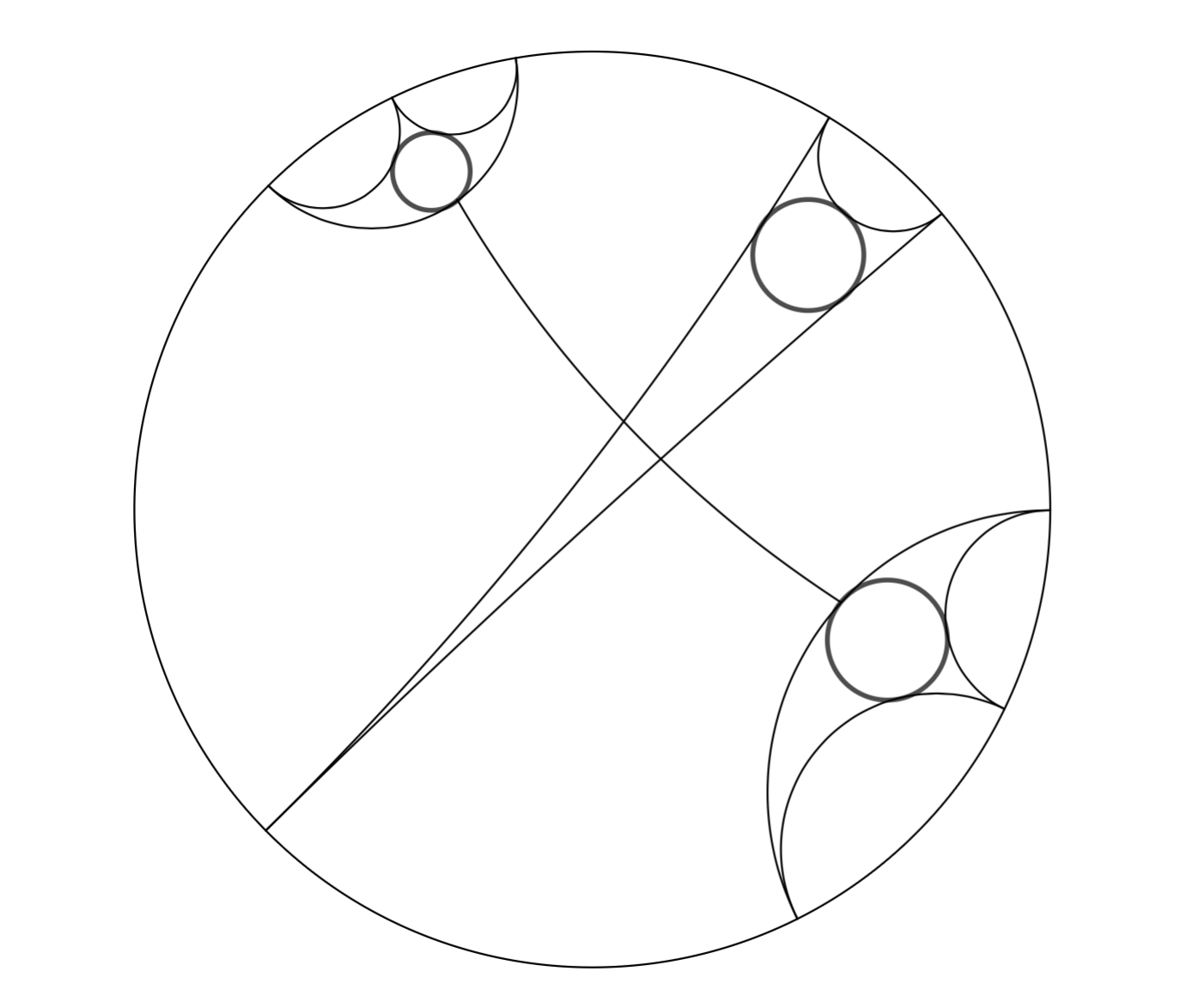}
    \caption{Lifting a geodesic to $\H^2$ (done with Geogebra)}
    \label{fig:equivmap}
\end{figure}
%
%(Notice that since $\lambda$ is filling, the discs interior to the $C_x$ have disjoint neighborhoods, and the minimal hyperbolic distance between two of them can be read in $S$.)
%
The covering $\tilde{S}^o$ is obtained from $\H^2$ by drilling out the interior of $C_x$ and gluing back a copy of $\R\P^\infty$ along $\R\P^1\simeq C_x$ for every $x\in V(T)$. By construction, the orbifold $S^o$ is homotopic to the quotient $\tilde{S}^o/\pi_1(S)$. 

We now proceed to the construction of an equivariant map $F \colon T^o\to \tilde{S}^o$. There is already an identification between the orbifold parts of both spaces, so that we are left to define the map $F$ on the connecting part. 

For every pair $(x,y)\in V(T)^2$, we must define a path $F(x,y)$ in $\tilde{S}^o$ connecting the points of $C_x$ and $C_y$ identified to the end points $h_x,h_y$ of $(x,y)$ in $T^o$. A first guess would be to consider the geodesic path $\gamma$ between the points $h_x$ and $h_y$. This path actually projects to the geodesic joining $x$ to $y$ in $T$. However it may intersect a forbidden circle $C_z$, in which case it enters its circumscribed ideal triangle $\Delta_z$ by one side and leaves it by another. Call $p_z$ the ideal vertex at the intersection of these two sides. 
We can homotope $\gamma$ inside $\Delta_z$ to a path avoiding $C_z$ which stays on the side containing $p_z$, see Figure \ref{fig:equivmap}.

%\textcolor{red}{Pourquoi de ce côté et pas de l'autre ? Je ne penses pas qu'on ait le choix sur le côté, et d'ailleurs si on fait certaines alternances gauche droite entre les triangles il faut couper le cercle inscrit...}
%
%This path will indeed belong to $f^{-1}([x,y])$ where $[x,y]$ denotes the geodesic joining $x$ to $y$ and will avoid the inscribed circles included in that region. 
%
%Let $U$ be the set of vertices lying between $x$ and $y$. These vertices belong to two disjoint classes $U=U_L\cup U_R$: a branch point $u$ belongs to $U_L$ (resp. $U_R$) if the branch leaving $u$ turns left when moving from $x$ to $y$. We define $F(x,y)$ to be any transverse path in $f^{-1}([x,y])$ which separate the collection of circles $U$ accordingly to this decomposition. 

Moreover, we can choose those paths in such a way that $F$ is $\pi_1(S)$-equivariant. Let us now consider a triple of points $x,z,y$ lying on a geodesic of $T$ in that order. We have defined $F(x,z)$, $F(z,y)$ and $F(x,y)$: it is not hard to see that the region enclosed by the three arcs and the boundary of $C_z$ does not contain any other circle, hence it can be filled by a triangle: this extends $F$ to the $2$-skeleton of $T^o$. 
This procedure can be continued to define an equivariant map $F:T^o\to \tilde{S}^o$, which induces a map $\overline{F}:\Sigma^o\to S^o$.

We would like to show that $\overline{F}$ is a homotopy equivalence. The space $\tilde{S}^o$ is Eilenberg-MacLane and Lemma \ref{retraction} below shows that so is $T^o$, hence it is sufficient to prove that $\overline{F}$ induces an isomorphism between fundamental groups.
Behold the following commutative diagram, and observe that the five lemma reduces the statement to showing that $F_*$ is an isomorphism.
$$\xymatrix{
0\ar[r]& \pi_1(T^o)\ar[d]^{F_*}\ar[r] & \pi_1(\Sigma^o)\ar[r]\ar[d]^{\overline{F}_*}& \pi_1(S)\ar[d]^=\ar[r]& 0 \\
0\ar[r]& \pi_1(\tilde{S}^o)\ar[r] & \pi_1(S^o)\ar[r]& \pi_1(S)\ar[r]& 0 
}$$
This last statement is clear from the fact that $\pi_1(T^o)$ and $\pi_1(\tilde{S}^o)$ are both isomorphic to a free product of copies of $\Z/2\Z$ indexed by $V(T)$ (see again Lemma \ref{retraction}). 
%As $T^o$ is an Eilengerg-Maclane space, the same is true for $\Sigma^o$ hence it is sufficient to show that $\pi_1(\Sigma^o)$ and $\pi_1(S^o)$ are isomorphic groups. 
%Fix a hyperbolic structure on $S$ and represent $T$ as the dual tree of a measured lamination $\lambda$ on $S$. Denote by $f$ the quotient map $f:\tilde{S}\simeq \H^2\to T$. 
%Denote by $\tilde{S}^o$ the covering of $S^o$ corresponding to the morphism $\pi_1(S^o)\to \pi_1(S)$. 
%It is sufficient to show that there is an isomorphism $\pi_1(\tilde{S}^o\simeq \pi_1(T^o)$ which is $\pi_1(S)$-equivariant up to conjugation (détailler...). 
%For any finite subset $W\in V(T)$, we can embedd the tree $T'_W$ inside $\H_2$ by sending each cycle to the corresponding inscribed circle. It suffices to lift the edges between closest neighbours to arcs transversal to the lamination...
\end{proof}

\subsection{Homology of $T^o$}

The homology of $T^o$ can be computed from its finite sub-complexes, which are easy to understand thanks to the following lemma. For a finite set $W\subset V(T)$, let $T^o_{(W)}$ to be the union of cells involving $W$ only: a cell belongs to $T^o_{(W)}$ when all its $0$-faces are of the form $(x,h)$ for $x\in W$. 
We define $T^o_W$ to be the sub-complex of $T^o_{(W)}$ whose connecting part reduces to the $1$-cells $(x,y)$ for $x,y \in W$ such that there is no other element in $W$ on the geodesic joining them. In more intuitive terms, $T^o_W$ is a collection of $\R\P^\infty$ indexed by $W$, connected in a tree-like fashion given by the embedding of $W$ in $T$. 

\begin{lemma}\label{retraction}
For all finite $W\subset V(T)$, the cell-complex $T^o_{(W)}$ retracts by deformation on $T^o_W$ .
\end{lemma}

\begin{proof}
We define the retraction by induction on the maximal dimension of the truncated simplices $\Delta_U\subset T^o_{(W)}$. % with maximal dimension, and thereby decrease the dimension down to $1$. 
Let $U =\{x_0,\dots,x_n\}$ correspond to one of them, it is the intersection of $W$ with a geodesic in $T$. We retract $\Delta_U$ by deformation onto the union of $\Delta_{U'}$ for $U'\subset U$ ranging over all subsets which do not contain both $x_0$ and $x_n$. % The number of cells $\Delta_U$ with dimension $>1$ has decreased, and we keep going until all such truncated simplices have dimension one.
This procedure stops when $U=\{x,y\}$ and $x,y$ are closest neighbours in $W$.
\end{proof}

%This lemma shows that $\pi_k(T^o)=0$ for all $k>1$. Indeed, any map $f\colon \mathbb{S}^k\to T^o$ lands into a subcomplex $T^o_W$ for some $W$, which retracts on a finite bouquet of $\R\P^\infty$, whose higher homotopy groups are trivial.

%The group $\pi_1(S)$ acts freely and properly on $T^o$, and we write $\Sigma^o=T^o/\pi_1(S)$.

%In this technical section, we set $\pi=\pi_1(S)$. %Observe that $\pi$ acts on the cell-complex $T^o$, freely and properly. The group $\Pi=\pi_1(T^o/\pi)$ fits in the exact sequence:
%\begin{equation}\label{ext-arbre}
%1\to \pi_1(T^o)\to \Pi\to \pi\to 1.
%\end{equation}
%
% We define a $1$-cochain $\rho\in C^1(T^o,\{\pm 1\})$ in the following way: $\rho(h,k)=-1$ for any half-edges $h,k$ incident to a same vertex. The other $1$-cells come from truncated simplices $\Delta_U$ and can be of two types: for $1$-cells $e$ joining different vertices in $U$ we put $\rho(e)=-1$ and for the others, we define $\rho$ consistently with the previous definition. Hence $\rho(e)=-1$ if and only if their ends are attached to different half-edges. 
%
We define a $1$-cochain $\rho\in C^1(T^o,\{\pm 1\})$ sending every $1$-cell of $T^o$ to $-1$. It is a cocycle because the $2$-cells of $T^o$, being either hexagons (orbifold part) or squares (contained in some $\Delta_W$ for $W$ of cardinal $3$), have an even number of $1$-faces. %a cocycle since all $2$-faces of truncated simplices $\Delta_U$ have an even number of edges.%, and being invariant under $\pi$ it descends to a class in $\rho \in Z^1(S^o,\{\pm 1\})$. 
%
%This corresponds to a homomorphism $\rho \colon \Pi \to \{\pm 1\}$, and we denote $\R^-$ the $\Pi$-module defined by $\gamma x=\rho(\gamma)x$.
The geometric idea underlying this definition is that any half-edge stands for a local coorientation of the lamination $\tilde{\lambda}$, say pointing to the closest singular point. Following an edge $e$ in $T^o$ (transverse to $\tilde{\lambda}$), we arrive at the other end with the opposite co-orientation, giving $\rho(e)=-1$.

This cocycle defines a homomorphism $\rho:\pi_1(T^o)\to \R$ and we denote by $\R^-$ the vector space $\R$ with the action $\gamma.x=\rho(\gamma)x$. Our first task is to compute the homology of $T^o$ with coefficients in $\R$ and $\R^-$. 

\begin{lemma}\label{homology-arbre}
We have $H_k(T^o,\R)=0$ if $k\ne 0$, $H_k(T^o,\R^-)=0$ if $k\ne 1$ and $$H_1(T^o,\R^-)\simeq \Bb(T).$$ 

%The vector space $H_1(T^o,\R^-)$ is isomorphic to the space $\Bb(T)$ generated by pairs $(x,y)\in V(T)^2$ with $x\ne y$ modulo the relations $(y,x)=(x,y)$ and $(x,z)=(x,y)+(y,z)$ whenever $y$ belongs to the geodesic from $x$ to $z$.
%\item[-] $H_2(T^o,\R^-)\simeq \R^{(V(T))}$.
%\item[-] $H_k(T^o,\R^-)=0$ for $k>2$
%\end{enumerate}
\end{lemma}

\begin{proof}
Observe that $T^o=\varinjlim T^o_{(W)}$ as $W$ exhausts the finite subsets of the countable set of branch points $V(T)$ and by Lemma \ref{retraction}, $T^o_{(W)}$ retracts by deformation on $T^o_W$, 
thus $H_*(T^o,\R^\pm)=\varinjlim H_*(T^o_W,\R^\pm)$.

%From Lemma \ref{retraction}, we know that $T^o_W$ retracts by deformation on a CW-complex $T'_W$ which is a union of $\R\P^\infty$ lying above each $x\in W$ connected by an edge whenever $x$ and $y$ are such that there is no other point $z\in W$ lying between $x$ and $y$ in $T$. Recall that the retraction $T^o_W \to T'_W$ was relative to the $1$-skeleton, so it is $\rho$-equivariant and identifies homology groups with values in $\R^-$ as well as those in $\R$.
We may forget about the cocycle $\rho$ while computing the untwisted real homology, and further retract the space $T^ o_W$ on a wedge of infinite projective spaces. Thus $H_0(T^o_W,\R)=\R$ and $H_k(T^o_W,\R)=0$ for $k>0$ so the same goes for $T^o$. % by taking the inductive limit.

We now return to the twisted homology of $T^o_W$. For this we consider the double cover $T'_W\to T^o_W$ corresponding to $\rho$ and compute the untwisted homology of the total space: it splits into the $\pm 1$-eigenspaces of the involution which coincide with $H_*(T^o_W,\R^\pm)$ respectively. 
The space $T'_W$ is homotopy equivalent to a graph with vertex set $W$, and two edges above each edge $e$ of $T^o_W$ connecting its end points with opposite orientations, as shown below.

\[\xymatrix{\bullet \ar@/_/[dr]& & &\bullet\ar@/_/[dl] \\
&x\ar@/_/[r]\ar@/_/[dl]\ar@/_/[ul]&y\ar@/_/[l]\ar@/_/[ru]& \\
\bullet\ar@/_/[ur] & & & }\]

It follows that $H_0(T'_W,\R)=\R$ and $H_k(T'_W,\R)=0$ if $k>1$.
Moreover $H_1(T'_W,\R)$ has a basis formed by the cycles $c(x,y)\in H_1(T'_W,\R)$ indexed by the edges $(x,y)$ of $T'_W$, which consist in making a round trip from $x$ to $y$, following the arrows. The Galois involution of $T'_W$ exchanges the orientation of $c(x,y)$, so $H_1(T^o_W,\R^-)$ is freely generated by pairs $(x,y)$ where $x,y$ are closest neighbours in $W$. 

Taking the limit as $W$ converges to $V(T)$, we obtain $H_k(T^o,\R^-)=0$ for $k=0$ and $k>1$. If an edge $(x,y)$ gets subdivided into $(x,z)$ and $(z,y)$ as $W$ increases, we have $c(x,y)=c(x,z)+c(z,y)$ which is compatible with the equality $(x,y)=(x,z)+(z,y)$, and provides the desired isomorphism for the inductive limit of $H_1(T^o_W,\R^-)$.
\end{proof}

\subsection{Homology of the quotient $\Sigma^o=T^o/\pi$}\label{homology-quotient}

Let us write $\pi=\pi_1(S)$ for short. The cocycle $\rho$ on $T^o$ is $\pi$-invariant, so it induces a morphism $\pi_1(\Sigma^o)\to \{\pm 1\}$ that we also denote $\rho$. The $\pi$-equivariant homotopy equivalence between $\Sigma^o$ and $S^o$ thus yields a morphism $\pi_1(S^o)\to\{\pm 1\}$. By the remark following Lemma \ref{retraction}, this morphism is the co-orientation monodromy of $\lambda$, so its kernel corresponds to the covering $S'\to S^o$.
Consequently, we may deduce the homology of $S^o$ with coefficients in $\R^{\pm}$ from that of $\Sigma^o$ with the same coefficients.

The $2$-fold covering $S'$ of $S^o$ ramified over $R$ satisfies $\chi(S')=2\chi(S)-(4g-4)=8-8g$ by the Riemann-Hurwitz formula. As $H_*(S^o,\R^\pm)=H_*(S',\R)^\pm$ we get that $H_*(S^o,\R)=H_*(S,\R)$ whereas $H_k(S^o,\R^-)=0$ if $k\ne 1$ and $\dim H_1(S^o,\R^-)=6g-6$.

On the other hand, we can compute $H_*(T^o/\pi,\R^{\pm})$ from $H_*(T^o,\R^{\pm})$ using the Cartan-Leray spectral sequence. Its second page is $E^2_{p,q}=H_p(\pi,H_q(T^o,\R^{\pm}))$ and converges to $H_{p+q}(\Sigma^o,\R^{\pm})$. Lemma \ref{homology-arbre} implies that, with both coefficients, the second page has only one line, whence the isomorphisms:
$$H_*(\Sigma^o,\R)=H_*(\pi,\R)=H_*(S,\R),\quad H_*(\Sigma^o,\R^-)=H_{*-1}(\pi,\Bb(T)).$$ 
This yields $H_1(\Sigma^o,\R^-)=H_0(\pi,\Bb(T))=\Bb(T)_\pi$ and $H_k(\pi,\Bb(T))=0$ for $k=1,2$. Observe that from Poincaré duality we get $H_2(\pi,\Bb(T))=H^0(\pi,\Bb(T))=\Bb(T)^\pi=0$. It is not surprising that $\Bb(T)$ has no invariant cycles as $\pi$ acts freely on $V(T)$. We do not have a similar explanation for the vanishing of $H^1(\pi,\Bb(T))$.

\subsection{Intersection form}

In the commutative diagram below, the first colon is a Galois covering of surfaces with group $\pi$. We have the identifications $H_1(\tilde{S}',\R)^-=H_1(\tilde{S}^o,\R^-)=H_1(T^o,\R^-)=\Bb(T)$ and $H_1(S',\R)^-=H_1(S^o,\R^-)=\Bb(T)_\pi$.

\[\xymatrix{ \tilde{S}'\ar[r]^{\tilde{p}}\ar[d]^\pi & \tilde{S}^o\ar[d]^\pi \\ S'\ar[r]^{p}& S^o}\]

\begin{proposition}\label{symplectomorphismes}
The isomorphisms $H_1(\tilde{S}',\R)^-=\Bb(T)$ and $H_1(S',\R)^-=\Bb(T)_\pi$ preserve the symplectic forms. 
\end{proposition}

\begin{proof}
Let us begin with the first isomorphism. Recall that we defined an equivariant map $F:T^o\to \tilde{S}^o$: it sends the cell $(x,y)$ to a path $F(x,y)$ joining the orbifold points corresponding to $x$ and $y$ and avoiding all other orbifold points. As the homology of the orbifold part of $T^o$ with coefficients $\R^-$ vanishes identically, these paths actually define cycles in $H_1(\tilde{S}^o,\R^-)$, which in $H_1(\tilde{S}',\R)$ are represented geometrically by $c(x,y)=\tilde{p}^{-1}(F(x,y))\subset \tilde{S}'$. Notice that these cycles have a natural orientation (given by the co-orientation of the lifted lamination $\tilde{\lambda}'$).   

%To explain why the pairing coincides with the intersection on $H_1(\tilde{S}^o,\R^-)=H_1(\tilde{S}',\R)^-$, we need a few properties of the equivariant map $F:T^o\to \tilde{S}^o$ constructed in Lemma \ref{equiv-homotopy}.
%
%This map sends the geodesic segment $xy$ to a path $F(x,y)$, which in turn lifts to a cycle $c_{(x,y)}$ in the ramified double cover of $\tilde{S}^o$, and the intersection of those cycles corresponds to the pairing on $H_1(\tilde{S}^o,\R^-)$.
%
Recalling the definition of the pairing in $\Bb(T)$ given in Proposition \ref{prop-pairing}, it suffices to compute $c(x,y)\cdot c(z,t)$ in the case where $(x,y)$ and $(z,t)$ are disjoint or consecutive.

In the first case, the cycles $c(x,y)$ and $c(z,t)$ are also disjoint so their intersection vanishes.
In the second case, the cycles $c(x,z)$ and $c(z,y)$ only intersect in a neighborhood of $z$ which looks like in the right hand side of Figure \ref{fig:etoile}. The lifted cycles $c(x,z)$ and $c(z,y)$ go straight through the intersection point, oriented as shown. Analyzing the two possible cases, we find that the signs coincide.
\begin{figure}[htbp]
    \centering
    \includegraphics[width=4cm]{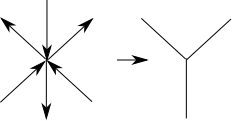}
    \caption{Double covering over a branching point}
    \label{fig:etoile}
\end{figure}

Let us now consider the quotient. We showed in Section \ref{homology-quotient} that $H_1(S',\R)=H_1(\tilde{S}',\R)_\pi$. The result follows from the fact that the intersection form on $H_1(S',\R)$ coincides with the averaged intersection form on $H_1(S',\R)$. 
\end{proof}

% \newpage

\bibliographystyle{alpha}
\bibliography{biblio-tex.bib}

\end{document}